\documentclass[12pt,reqno]{amsart}
\usepackage{color,mathrsfs,epsfig}
\usepackage{psfrag,graphicx}
\usepackage{amssymb}

\newcommand{\R}{\mathbb{R}}
\newcommand{\F}{{\mathcal F}}
\newcommand{\G}{{\mathcal G}}

\newcommand{\eps}{\varepsilon}
\newcommand{\p}{\partial}
\renewcommand{\l}{\lambda}
\newcommand{\Leb}{{\mathscr L}}

\renewcommand{\div}{{\rm div}\,}
\newcommand{\divtx}{{\rm div}_{t,x}\,}

\newcommand{\weaks}{\stackrel{*}{\rightharpoonup}}

\newcommand{\loc}{{\rm loc}}

\newcommand{\ul}{\underline \eta}
\newcommand{\ulq}{\underline q}

\newtheorem{theorem}{Theorem}[section]
\newtheorem{lemma}[theorem]{Lemma}
\newtheorem{propos}[theorem]{Proposition}
\newtheorem{corol}[theorem]{Corollary}

\theoremstyle{definition}
\newtheorem{definition}[theorem]{Definition}

\theoremstyle{remark}
\newtheorem{remark}[theorem]{Remark}

\newtheorem{question}[theorem]{Question}

\numberwithin{equation}{section}

\begin{document}
\bibliographystyle{plain}
\title[Well-posedness of continuity equations and applications]{Some new well-posedness results for continuity and transport equations, and applications to the chromatography system}
\author{Luigi Ambrosio}
\address{L.A.: Scuola Normale Superiore,
Piazza dei Cavalieri 7, 56126 Pisa, Italy}
\email{l.ambrosio@sns.it}
\author{Gianluca Crippa}
\address{G.C.: Dipartimento di Matematica,
Universit\`a degli Studi di Parma,
viale G.P.~Usberti 53/A (Campus),
43100 Parma, Italy}
\email{gianluca.crippa@unipr.it}
\author{Alessio Figalli}
\address{A.F.: Centre de Math\'ematiques Laurent Schwartz, \'Ecole Polytechnique,
91128 Palaiseau, France}
\email{figalli@math.polytechnique.fr}
\author{Laura V.~Spinolo}
\address{L.V.S.: Centro De Giorgi, Collegio Puteano, Scuola Normale Superiore,
Piazza dei Ca\-va\-lie\-ri 3, 56126 Pisa, Italy}
\email{laura.spinolo@sns.it}

\begin{abstract}
We obtain various new well-posedness results for continuity and
transport equations, among them an existence and uniqueness theorem
(in the class of strongly continuous solutions) in the case of
nearly incompressible vector fields, possibly having a blow-up of
the $BV$ norm at the initial time. We apply these results (valid in
any space dimension) to the $k \times k$ chromatography system of
conservation laws and to the $k \times k$ Keyfitz and Kranzer
system, both in one space dimension.
\end{abstract}

\maketitle

\section{Introduction}
\subsection{Continuity and transport equations}

In the last few years some progress has been made on the well-posedness of the continuity and transport
equations
$\partial_t \rho+\div (b\rho)=0$, $\partial_t w+b\cdot\nabla w=0$,
for weakly differentiable (in the space variables) velocity fields $b$.
The first seminal paper by DiPerna and Lions \cite{DiPLi} considered Sobolev vector fields with spatial divergence
$\div b$ in $L^1([0,T];L^\infty)$, and more recently in \cite{Amb:trabv} Ambrosio
extended the result to $BV$
vector fields, assuming absolute continuity of the divergence and $L^1([0,T];L^\infty)$ regularity of
its negative part (see also the lecture notes by Ambrosio and Crippa \cite{ambcri:ex}).
This bound is the simplest way, as it can be easily seen with the method
of characteristics, to prevent blow-up in finite time of the solutions; on the other hand, it was
already clear from the application made in \cite{ABDL,ambdel} to the Keyfitz and Kranzer
system \cite{KK}, that this assumption is not very natural. The Keyfitz and Kranzer system has the peculiar form
$$
\partial_t u+\div\bigl(uf(|u|)\bigr)=0
$$
and (at least formally) it decouples in a scalar conservation law for $\rho:=|u|$, namely
$$\partial_t\rho+\div(\rho f(\rho))=0\,,$$ and a transport equation for the ``angular''
part $\theta$, related to $u$ by $u=\rho\theta$:
$$
\partial_t\theta+ f(\rho)\cdot\nabla\theta=0\,.
$$
If we take $\rho$ as the entropy admissible solution of the scalar
conservation law, it turns out that the velocity field $b=f(\rho)$
in the transport equation belongs to $BV$ if the initial condition
$\bar\rho$ is in $BV$, but its distributional divergence needs not
be absolutely continuous; the only (weaker) information available
is that there is a bounded function, namely $\rho$, transported by
$b$. This leads to the concept of nearly incompressible vector
fields (namely those $b$ for which a bounded transported density
$\rho$ exists), and well-posedness results in this class of vector
fields have been investigated in
\cite{Amb-De-Ma,CDLest,del:notes}, also in connection with a
compactness conjecture made by Bressan \cite{Bre:illposed}. See
also \cite{CDL} for a counterexample to the applicability of these
techniques to general multidimensional systems of conservation
laws.

In this paper we obtain a well-posedness result for this class of
vector fields, see Theorem~\ref{t:uni1} for a precise statement.
However, the main result is an improvement of the
$L^1([0,T];BV_\loc)$ condition considered in all previous papers:
going back to the case of the Keyfitz and Kranzer system, it is
clear that we cannot expect this regularity for $\rho$ (and then
for $b=f(\rho)$), unless $\bar\rho\in BV$. More precisely, in the
one-dimensional case and imposing suitable conditions on the flux
function $f$, if $\bar\rho\in L^\infty$ we gain indeed a
regularizing effect (see Ole{\u\i}nik~\cite{Ole}). However, this
provides only $L^1_\loc(]0,T];BV_\loc)$ regularity, and also
(using the equation satisfied by $\rho$) $BV_\loc(]0,T]
\times\R^d)$ regularity. This last assumption on the vector field
is critical in view of a non-uniqueness example provided by Depauw
\cite{Depauw:ex}, where the vector field has precisely this
regularity. Our main result, given in Theorem~\ref{t:uni2}, is
that for this class of vector fields (adding bounds on the
divergence in the same spirit of the nearly incompressibility
condition, see Definition~\ref{d:nearly}) existence and uniqueness
can be restored, provided one works with strongly continuous in
time solutions. Here, of course, the main difficulty is in the
existence part, since standard approximation schemes in general do
not provide this strong continuity property.

The main application of our results concerns the so-called chromatography system. As discussed in Section~\ref{s:compar}, we obtain new well-posedness theorems for this equation. As a byproduct of our analysis, we also have applications to the one-dimensional Keyfitz and Kranzer system, obtaining in particular a well-posedness
theorem for bounded initial data (results in this flavor were
already known: see for example {Freist{\"u}hler~\cite{freistuhler} and Panov~\cite{Panov:onthe}). 

\subsection{Chromatography}\label{s:compar}
Although all our discussion and results can be easily extended to
the $k \times k$ chromatography system (see Remark \ref{rmk:k
times k}), for simplicity of exposition we will always consider
only the case $k=2$. The chromatography equation
\begin{equation}\label{e:chrom}
\begin{cases}
\partial_t u_1 + \partial_x \left( \displaystyle \frac{u_1}{1+u_1+u_2} \right) = 0 \\ \\
\partial_t u_2 + \partial_x \left( \displaystyle \frac{u_2}{1+u_1+u_2} \right) = 0
\end{cases}
\end{equation}
is a system of conservation laws belonging to the so-called Temple class. Consider first a
general system of conservation laws
$$
 \partial_t  U+ \partial_x F (U) = 0 \, ,
\qquad \text{where $(t, x) \in [0, +\infty[ \times \R$ and $U \in \R^k$.}
$$
The associated Cauchy problem is well-posed under the assumption
that the total variation of the initial datum $\bar U$ is sufficiently small (see e.g.~\cite{Bre:book} and the
references therein). If the total variation of $\bar U$ is bounded but large, then the solution may experience blow
up in finite time, as shown for instance by Jenssen~\cite{Jenssen}.

Temple systems were introduced in~\cite{temple} and they are defined by special properties imposed on the structure
of the eigenvector fields of the Jacobian matrix $DF(U)$. Thanks to these features, well-posedness results for Temple
systems are available for a much larger class of initial data compared to general systems of conservation laws.
In particular, Serre \cite{serre} obtained global existence of weak solutions for $2\times 2$ Temple systems with no
smallness assumptions on the total variation of the initial datum. By relying on wave front-tracking techniques,
in \cite{baitibressan}  Baiti and Bressan constructed a Lipschitz continuous semigroup defined on initial data with large
but bounded total variation. The extension of the semigroup to $L^\infty$ initial data was first achieved by Bressan
and Goatin~\cite{bressangoatin} under a nonlinearity assumption on the structure of the eigenvector fields of the
Jacobian matrix $DF(U)$, and then extended by Bianchini in \cite{bianchini} to general Temple class systems, by removing any
nonlinearity assumption. One of the main difficulties faced in \cite{bianchini} is that if the nonlinearity
assumption fails, one cannot hope for a Lipschitz dependence of the semigroup on the initial data. This is shown
by the example discussed in \cite{bressangoatin}, which involves precisely a family of Cauchy problems for the
chromatography system.

The main result in~\cite{bianchini} is the existence of a semigroup $S_t$ of weak solutions defined on
$L^\infty$ initial data and satisfying the following conditions:
\begin{enumerate}
\item $(t ,\bar U) \mapsto S_t \bar U$ is continuous with respect to $\R\times L^1_\loc$ topology, with values in
$L^1_\loc$;
\item If $\bar U$ is piecewise constant, then for $t$ sufficiently small $S_t \bar U$ coincides with the function
obtained by gluing together the ``standard'' solutions constructed as in Lax \cite{lax}.
\end{enumerate}
The proof exploits the so called wave front-tracking algorithm, which relies on the approximation of the initial datum with piecewise constant functions.
In~\cite{bianchini} uniqueness is also obtained: namely, it is shown that there exists a unique semigroup of weak solutions satisfying conditions (1) and (2) above. In addition, the weak solutions provided by the maps $t \mapsto S_t \bar U$
are automatically entropy admissible.

The results in~\cite{bianchini} apply to general systems in the Temple class, while in the present work we
restrict to the chromatography system. The main novelties here are the following.

First, our approach is completely different from the one in~\cite{bianchini}.  Namely, by introducing a change of variables in analogy to Ambrosio and De Lellis~\cite{ambdel} and to Bressan and Shen~\cite{bressanshen},
we split the chromatography system in the coupling between a scalar conservation law and a transport equation, and then we heavily exploit transport equation techniques.

The second point concerns uniqueness: we manage to get it in the classes of functions defined in Section~\ref{s:chrom} without requiring any stability with respect to perturbations in the initial data. Conversely, the approach in~\cite{bianchini} is in the spirit of the Standard Riemann Semigroup (see
Bressan~\cite{bressansrs}): one shows that there exists a unique semigroup of weak solutions satisfying (1) and (2) above, thus one needs to prove stability and the ``right'' behavior
on piecewise constant data to get uniqueness. This does not correspond, a priori, to requiring uniqueness of
weak entropy admissible solutions. Finally, we point out in passing that we manage to relax the hypothesis
of strict hyperbolicity made in~\cite{bianchini}, which fails when the initial data $\bar u_1$ and $\bar u_2$
vanish in some region (see Definition~\ref{d:F}). However, the price we have to pay in this case is the requirement
$\bar u_1 + \bar u_2 \in BV_\loc(\R)$.

The idea of attacking the chromatography system via a change of variables which splits the system into a coupling
between a scalar conservation law and a transport equation came from Bressan and Shen \cite{bressanshen} (see also Panov~\cite[Remark 4]{Panov:onthe}).
In~\cite{bressanshen} the authors are mainly concerned with the study of ODEs with discontinuous vector fields:
they formulate conditions under which uniqueness holds, and they show that the vector field obtained from the
chromatography system after their change of variables satisfies these conditions,
provided the initial data take value in a particular range. In our proof we perform a different change of
variables, which has the advantage of being linear, and hence behaves well under weak convergences.
In particular, we can come back to the original system, obtaining distributional solutions which are
admissible in the sense described in the Appendix. Also, we manage to handle a larger class of initial data.

\smallskip\noindent
{\bf Acknowledgement.} We thank Fabio Ancona and an anonymous
referee for pointing out to us the extension of our results to
$k\times k$ systems, as illustrated in Remark~\ref{rmk:k times k}.

\subsection{Content of the paper}

In Section~\ref{s1} we prove the basic well-posedness results for
the continuity equation, first in a class of bounded nearly
incompressible vector fields having $BV_\loc([0,T]\times\R^d)$
regularity, and then under additional assumptions on the
transported density $\rho$ (see Definition~\ref{d:nearly}) for
bounded vector fields having $BV_\loc(]0,T]\times\R^d)$
regularity. In Section~\ref{s2} we discuss Depauw's example
\cite{Depauw:ex} and its relation with our well-posedness results.
In Section~\ref{s:chrom} we present an application of these
results to the chromathography system. First we classify in
Lemma~\ref{l:entropy} the entropy-entropy flux pairs for this
system, and then we prove two basic existence and uniqueness
results: the first one, Theorem~\ref{t:chrom1}, provides existence
and uniqueness in $L^\infty$ under the assumption that the sum
$\bar u_1+\bar u_2$ of the initial conditions is in $BV_\loc$; the
second one, Theorem~\ref{t:chrom2}, replaces this regularity
condition with $\inf_K\bar u_1+\bar u_2>0$ for any $K\subset\R$
compact. The applications to the Keyfitz and Kranzer system are
discussed in Section~\ref{s:kk}: in Theorem~\ref{t:kk1} we show
existence and uniqueness of strongly continuous renormalized
entropy solutions, when the initial datum $\bar u$ is bounded and
satisfies $\inf_K | \bar u| >0$ for any $K\subset\R$ compact. In
the Appendix we list a few basic facts on scalar conservation laws
needed in the paper.

\subsection{Main notation and conventions}

Finally, we specify some (standard) convention about spaces of time-dependent functions: if $J\subset\R$
is an interval and $X$ is a separable Banach space, by $L^p(J;X)$ (resp. $L^p_\loc(J;X)$) we mean all
measurable functions $u:J\to X$ such that $\|u\|_X\in L^p(J)$ (resp. $L^p(I)$ for all $I\subset J$ compact).
In the particular cases of non-separable spaces $X$, in this paper $X=L^\infty$ or $X=BV$, the definition is the
same, but measurability is understood in a weak sense, using the embedding of these spaces in $L^1_\loc$.

We shall often consider locally bounded (in space and time) distributional solutions to the continuity equation
$\partial_t\rho+\div(b\rho)=0$, with $b$ locally integrable. In this case it is well known  (see for
example Lemma~1.3.3 in \cite{Daf:book}) that the map $t\mapsto\rho(t , \cdot)$ has a unique locally
weakly-$\ast$ continuous representative (i.e. $t\mapsto\int\rho(t,x)\phi(x)\,dx$ is continuous for any
$\phi\in L^\infty$ with compact support) and we shall always work with this representative, improving
in some cases the continuity from weak to strong. We shall use the notation $L^{\infty}-w^{\ast}$,
$L^1-s$ for the weak-$\ast$ and strong topologies, and $L^{\infty}_\loc-w^{\ast}$,  $L^1_\loc-s$ for their local counterparts. Finally, we will often consider the set of functions which are continuous
in time with values either in $L^{\infty}(\R^d)-w^{\ast}$ or in $L^1(\R^d)-s$
(or in their local counterparts).
These spaces are denoted by $\mathcal C^0 \big( [0, + \infty[ ; L^{\infty}(\R^d)-w^{\ast} \big)$
and $\mathcal C^0 \big( [0, + \infty[ ; L^1(\R^d)-s \big)$, respectively.

\section{Well-posedness results for continuity equations}\label{s1}

In this section we study the well-posedness of the continuity equation
\begin{equation}\label{e:cauchy}
\begin{cases}
\partial_t u(t,x) + \div \big( b(t,x) u(t,x) \big) = 0 \\
u(0,x) = \bar u (x) \,.
\end{cases}
\end{equation}
Our main results are presented in Theorems~\ref{t:uni1} and \ref{t:uni2}.

\subsection{A preliminary renormalization lemma}

We first prove a technical lemma of standard flavour, regarding renormalization and strong continuity for solutions $u$ of equation \eqref{e:part} with a $BV$ nearly incompressible vector field, see De Lellis \cite{del:notes} for a systematic treatment of this topic. The main difference is that, for our subsequent discussion, we need to consider functions $q$ which are possibly not bounded away from zero, or even equal to zero on a non-negligible set.
\begin{lemma}\label{l:lift}
Let $b \in BV_\loc \big( [0,+\infty[ \times\R^d ;\R^d \big)$ be a bounded vector field.
Assume that there exists a nonnegative locally bounded function
$$
q \in BV_\loc \big( [0,+\infty[ \times \R^d \big) \cap
\mathcal C^0 \big( [0, + \infty[ ; L^{\infty}_\loc (\R^d) - w^{\ast} \big)
$$
satisfying $|Dq(t,\cdot)|(B_R)\in
L^\infty_\loc\big([0,+\infty[\big)$ for all $R>0$ and
$$
\partial_t q + \div (bq) = 0
$$
in the sense of distributions in $]0,+\infty[  \times \R^d$. Let $u \in L^\infty_\loc \big( [0,+\infty[ \times\R^d)$ be a distributional solution of
\begin{equation}
\label{e:part}
           \partial_t (qu) + \div \big( bqu \big) = 0.
\end{equation}
such that the map $t \mapsto q u(t, \cdot)$ is weakly-$\ast$ continuous in $L^{\infty}_\loc (\R^d)$.
Then:
\begin{enumerate}
\item[(i)] $u$ is a renormalized solution, in the sense that the map
$$
     t \mapsto q(t, \cdot) \beta\big( u(t, \cdot) \big)
$$
provides a distributional solution of
\begin{equation}\label{e:lrin}
\partial_t \big( q \beta(u) \big) + \div \big( b q \beta( u) \big) = 0
\end{equation}
on $]0,+\infty[ \times \R^d$, for any function $\beta \in {\rm Lip}_\loc(\R)$.
\item[(ii)] If in addition the map
$$
t \mapsto q(t,\cdot)
$$
is strongly continuous from $[0,+\infty[$ with values in $L^1_\loc(\R^d)$, then also the map
$$
t \mapsto qu(t,\cdot)
$$
is strongly continuous from $[0,+\infty[$ with values in $L^1_\loc(\R^d)$.
\item[(iii)] Assume that $u_1$ and $u_2 \in L^\infty_\loc \big( [0,+\infty[ \times\R^d)$ are solutions of~\eqref{e:part}
such that there exists $\tau \ge 0$ satisfying
$$
q(\tau,x) u_1(\tau,x) = q(\tau,x) u_2(\tau,x)  \qquad \text{for a.e.~$x \in \R^d$.}
$$
Then for all $t \in [0,+\infty[$ we have
$$
q(t,x) u_1(t,x) = q(t,x) u_2(t,x) \qquad \text{for a.e.~$x \in \R^d$.}
$$
\end{enumerate}
\end{lemma}

\begin{proof}
Let us define $B(t,x):=(q(t,x),b(t,x)q(t,x))$.
Then $B \in BV_\loc\big([0,+\infty[\times \R^d ; \R \times \R^d \big)$ and is divergence free in space-time.
We introduce an artificial ``time variable'' $s \in[0,+\infty[$ and observe that the function
$$
v(s,t,x):=u(t,x)
$$
is a distributional solution of
$$
\p_s v + \divtx(B v)=0 \qquad
\text{on $]0,+\infty[_s \times ]0,+\infty[_t \times \R^d_x$.}
$$
Since $B(t,x)$ is an autonomous (i.e. independent of $s$)
divergence free vector field belonging to
$L^\infty \big([0,+\infty[_s ; BV_\loc([0,+\infty[_t \times \R^d_x; \R \times \R^d) \big)$, thanks to Ambrosio's
renormalization theorem (see \cite{Amb:trabv}, or \cite[Section 5]{Amb:cetraro}) the function
$\beta(v(s,t,x))$ solves
$$
\p_s \beta( v) + \divtx(B \beta(v))=0 \qquad \text{on
$]0,+\infty[ \times ]0,+\infty[ \times \R^d$,}
$$
or equivalently
$$
\p_t (q\beta(u)) + \div (bq \beta(u))=0\qquad \text{on
$]0,+\infty[ \times \R^d$.}
$$
This proves (i).

We now show (ii). Applying the result in (i) with $\beta(u) = u^2$, we have that
$q (t, \cdot) \big( u (t, \cdot) \big)^2$ provides a distributional solution of~\eqref{e:lrin}. Hence, there exists
$z \in \mathcal C^0 \big( [0, + \infty [ ; L^\infty_\loc (\R^d) - w^\ast \big)$
such that
$$
     z  (t, \cdot) = q (t, \cdot) \big( u (t, \cdot) \big)^2 \quad  \textrm{a.e.~in $\R^d$,}
$$
but only {\em for almost every} $t \ge 0$. Since $t \mapsto q(t,\cdot)$ is strongly continuous
with values in $L^1_\loc(\R^d)$, it follows that
$$
q(t,\cdot) z(t,\cdot) \in \mathcal C^0 \big( [0, + \infty [ ; L^\infty_\loc (\R^d) -w^\ast \big) \,.
$$
Notice that the difficulty in the proof of the strong continuity of $t \mapsto q(t,\cdot)u(t,\cdot)$ comes from the fact that it is not a priori obvious that the equality
$$
q(t,\cdot) z(t,\cdot) = \big( q (t, \cdot)  u (t, \cdot) \big)^2 \,,
$$
which is true {\em for almost every} $t \ge 0$, in fact holds {\em for every} $t \ge 0$. This will be shown by truncating the equation at an arbitrary time.

Let us fix $t_0 \in [0,+\infty[$: our goal is showing that
$t \mapsto qu(t,\cdot)$ is strongly continuous in $L^1_\loc(\R^d)$ at
$t_0$. Let us define the functions
$$
b_{t_0}(t,x):=
\left\{
\begin{array}{ll}
0&\text{if }t \in]-\infty, t_0]\\
b(t,x)&\text{if }t \in ]t_0,+\infty[\,,\\
\end{array}
\right.
$$
$$
q_{t_0}(t,x):=
\left\{
\begin{array}{ll}
q(t_0,x)&\text{if }t \in]-\infty, t_0]\\
q(t,x)&\text{if }t \in ]t_0,+\infty[\,,\\
\end{array}
\right.
$$
$$
u_{t_0}(t,x):=
\left\{
\begin{array}{ll}
u(t_0,x)&\text{if }t \in]-\infty, t_0]\\
u(t,x)&\text{if }t \in ]t_0,+\infty[\,\\
\end{array}
\right.
$$
and notice that $q(t_0,x)\in BV_\loc(\R^d)$ by the
weak continuity of $t\mapsto q(t,\cdot)$ and the uniform in time and
local in space bound on $|Dq(t,\cdot)|$. Thanks to the weak
continuity of $q$ and $qu$ we get
$$
\partial_t q_{t_0} + \div \big( b_{t_0}q_{t_0} \big) = 0 \quad \text{ and }\quad
\partial_t \big( q_{t_0} u_{t_0} \big)  + \div \big( b_{t_0}q_{t_0}u_{t_0} \big) = 0
$$
in the sense of distributions on $\R \times \R^d$. By applying the
same strategy as in the proof of (i) with the vector field
$B_{t_0}:=(q_{t_0},b_{t_0}q_{t_0})$, we deduce that
$$
\partial_t  \big( q_{t_0} u^2_{t_0} \big)  + \div \big( b_{t_0} q_{t_0} u^2_{t_0} \big) = 0\qquad \text{on
}\R \times \R^d
$$
in the sense of distributions. Combining this with the fact that
the map $t\mapsto q_{t_0}(t,\cdot) u^2_{t_0} (t,\cdot)$ is constant and equal to
$\big( q u^2 \big) (t_0,\cdot)$ for $t \leq t_0$, we easily deduce that
$$
q u^2 (t,\cdot) \weaks q u^2 (t_0,\cdot)
$$
weakly-$\ast$ in $L^\infty_\loc(\R^d)$ as $t \to t_0^+$, and thanks to the strong continuity
of $t \mapsto q(t,\cdot)$ this implies
\begin{equation}\label{e:strong}
q^2 u^2 (t,\cdot) \weaks q^2 u^2 (t_0,\cdot)
\end{equation}
weakly-$\ast$ in $L^\infty_\loc(\R^d)$ as $t \to t_0^+$.

From \eqref{e:strong} combined with the weak-$\ast$ continuity in $L^\infty_\loc(\R^d)$ of the map $t \mapsto qu(t,\cdot)$, we deduce that
$$
qu(t,\cdot) \to qu(t_0,\cdot)
$$
strongly in $L^1_\loc(\R^d)$ as $t \to t_0^+$. This proves the right continuity at $t_0$. The proof of the left continuity is analogous.

Finally, to show (iii), we observe that $v_1(s,t,x):=u_1(t,x)$ and
$v_2(s,t,x):=u_2(t,x)$ are both solutions of
$$
\p_s v + \divtx(B v)=0 \qquad \text{on }]0,+\infty[_s \times
]0,+\infty[_t \times \R_x^d.
$$
Hence, by Ambrosio's renormalization theorem, $|v_1 - v_2|$
solves
$$
\p_s \big( |v_1-v_2| \big) + \divtx \big( B |v_1-v_2| \big)=0 \qquad \text{on
}[0,+\infty[_s \times [0,\infty[_t \times \R_x^d,
$$
or equivalently
\begin{equation}
\label{eq:renorm} \p_t \bigl(q|u_1 - u_2|\bigr) + \div
\bigl(bq|u_1-u_2|\bigr)=0.
\end{equation}
Moreover, by considering the functions $b_{t_0}$, $q_{t_0}$ and
$(u_2-u_1)_{t_0}$ as in the proof of (ii), we can apply as above
the renormalization with $\beta(u)=|u|$ to obtain
$$
\p_t \bigl(q_{t_0}|(u_1 - u_2)_{t_0}|\bigr) + \div
\bigl(b_{t_0}q_{t_0}|(u_1-u_2)_{t_0}|\bigr)=0,
$$
which as above implies that the map
\begin{equation}
\label{eq:weak cont} t \mapsto q(t,\cdot) | u_1(t,\cdot) -
u_2(t,\cdot)| \in L^\infty_{\loc} (\R^d)
\end{equation}
is weakly-$\ast$ continuous. Hence, arguing as in~\cite[Lemma
3.17]{del:notes}, \eqref{eq:renorm} and \eqref{eq:weak cont} imply
$$
\int_{B_R(0)}q(t,x)|u_1(t,x) - u_2(t,x)|\,dx \leq
\int_{B_{R+|t-s|\|b\|_{\infty}}(0)}q(s,x)|u_1(s,x) - u_2(s,x)|\,dx
$$
for all $t,s>0$.
In particular, by setting $s=\tau$, we obtain the desired result.
\end{proof}

\begin{remark} Assume that in Lemma~\ref{l:lift} we replace the hypothesis of strong continuity of $t \mapsto q(t,\cdot)$ with the condition that for every $t$
$$
    \frac{1}{C} \leq q(t, x) \leq C \quad \textrm{a.e. $x \in \R^d$}
$$
for some constant $C>0$.
Then the map $t \mapsto u(t,\cdot)$ is strongly continuous in $L^1_\loc$. This is shown in De Lellis \cite[Corollary 3.14]{del:notes}.
\end{remark}

\subsection{Well-posedness of weakly continuous solutions}

In this section we prove a well-posedness result for~\eqref{e:cauchy}
in a function space adapted to the vector field $b$: this space depends
on an auxiliary function $p$ linked to $b$ by~\eqref{amb1}.

\begin{definition}\label{d:spazione} Let $\Omega$ be an open set in an Euclidean space and
let $p \in L^\infty_\loc \big( \Omega )$ be nonnegative.
We denote by $L^\infty(p)$ the set consisting of the measurable functions $w$ defined on
$\Omega$ such that there exists a constant $C$ satisfying
$$
|w (y) | \leq C p(y) \qquad \text{for a.e.~$y \in \Omega$.}
$$
The smallest constant $C$ will be denoted by $\|w\|_{L^\infty(p)}$.
\end{definition}

\begin{remark}\label{rmk:wp} Notice that, for any $w \in L^\infty (p)$, we necessarily have
$w=0$ a.e.~on $\{p=0\}$. Hence, by adopting the convention that $\frac{w(y)}{p(y)}=0$ if
$p(y)=0$, the quotient
$$
\frac{w(y)}{p(y)}
$$
is well-defined on $\Omega$ and is a bounded measurable function satisfying $(w/p)\cdot p=w$ a.e.
\end{remark}

\begin{theorem}\label{t:uni1} Let $b \in BV_\loc \big( [0,+\infty[ \times\R^d ;\R^d \big)$
be a bounded vector field. Assume that there exists a nonnegative
locally bounded function $p \in BV_\loc \big( [0,+\infty[ \times
\R^d \big)$ satisfying $|Dp(t,\cdot)|(B_R)\in
L^\infty_\loc\big([0,+\infty[\bigr)$ for all $R>0$ and
\begin{equation}\label{amb1}
\partial_t p + \div (bp) = 0
\end{equation}
in the sense of distributions. Then:
\begin{enumerate}
\item[(i)] For any initial datum $\bar u \in L^\infty(p(0,\cdot))$ there exists a solution {$u \in
{\mathcal C}^0 \big( [0,+\infty[ ; L^\infty(\R^d)-w^* \big) \cap
L^\infty(p)$} to the Cauchy problem~\eqref{e:cauchy}.
\item[(ii)] The solution $u$ to \eqref{e:cauchy} is unique in the class
${\mathcal C}^0 \big( [0,+\infty[ ; L^\infty(\R^d)-w^* \big) \cap
L^\infty(p)$.
\item[(iii)] $\|u(t,\cdot)\|_{L^\infty (p(t, \cdot))}\leq\|\bar u\|_{L^\infty(p(0, \cdot))}$ for all $t\geq 0$.
\end{enumerate}
\end{theorem}

\begin{proof}
Properties (i), (ii) and (iii) are trivially satisfied when $p$ is identically zero. Hence in the following, without any loss of generality, we assume that $p$ is not identically zero.
{\sc Existence.} Let us define $\l_0:= \bar u / p(0,\cdot)$ (see Remark~\ref{rmk:wp}). Roughly
speaking, the strategy to construct a solution is to solve
$$
\left\{
\begin{array}{l}
\p_t \l + b \cdot \nabla \l=0\\
\l(0,\cdot)=\l_0\,,
\end{array}\right.
$$
and to define $u=\l p$. Since $b$ is not smooth, we need a
regularization argument.

We first remark that the assumption that $p$ is nonnegative and
not identically zero implies that $p(t,\cdot)$ is not identically zero
for any $t > 0$. Indeed, assume by contradiction that
$$
p(\tau, x) =0 \qquad \textrm{for a.e. $x \in \R^d$}
$$
for some $\tau \ge 0$. Then by applying Lemma~\ref{l:lift}(iii) with  $q= p$, $u_1 \equiv 1$ and $u_2 \equiv 0$
we obtain that for every $t \ge 0$
$$
p(t, x) = 0  \qquad \textrm{for a.e. $x \in \R^d$},
$$
against our assumption.

We now consider a sequence $\eta_\eps$ of convolution kernels supported
on the whole $\R^d$ (Gaussian kernels, for instance), and we define
$$
p_\eps:=p*\eta_\eps \,, \qquad \qquad
b_\eps:=\frac{(bp)*\eta_\eps}{p*\eta_\eps}\,.
$$
Then $p_\eps$ is smooth and strictly
positive everywhere, so that $b_\eps$ is smooth and bounded. We now solve
for every $\eps > 0$ the Cauchy problem
$$
\left\{
\begin{array}{l}
\p_t \l_\eps + b_\eps \cdot \nabla \l_\eps=0 \\ \\
\l_\eps(0,\cdot)= \displaystyle \frac{\bar u*\eta_\eps}{p(0,\cdot)*\eta_\eps}\,.
\end{array}\right.
$$
Then, using the identity $\p_tp_\eps+\div (b_\eps p_\eps)=0$, one can easily
check that
$$
\left\{
\begin{array}{l}
\p_t (\l_\eps p_\eps) + \div( b_\eps \l_\eps p_\eps)=0 \\
\l_\eps p_\eps(0, \cdot)=\bar u*\eta_\eps\,.
\end{array}\right.
$$
Since
\begin{equation}
\label{eq:Linfty bound}
\|\l_\eps(t,\cdot)\|_{L^\infty(\R^d)} \leq \|\l_\eps(0,\cdot)\|_{L^\infty(\R^d)}
\leq \| \bar u \|_{L^\infty(p(0,\cdot))}\qquad \text{for every $t \in [0,+\infty[$,}
\end{equation}
the functions $\l_\eps$ are uniformly bounded in $L^\infty \big( [0,+\infty[ \times \R^d \big)$. Hence, up to subsequences, $\l_\eps$ converges weakly-$\ast$ in $L^\infty \big( [0,+\infty[ \times \R^d \big)$ to a function $\l \in L^\infty \big( [0,+\infty[ \times \R^d \big)$. Observing that $p_\eps \to p$ in $L^1_{\loc} \big( [0,+\infty[ \times \R^d \big)$, we get that
$$
\l_\eps p_\eps \weaks \l p: = u \qquad \text{ weakly-$\ast$ in
$L^\infty_{\loc} \big( [0,+\infty[ \times \R^d \big)$,}
$$
and $u$ solves \eqref{e:cauchy}. Moreover, thanks to \eqref{eq:Linfty bound}, (iii) holds.
Finally, the fact that $\l_\eps p_\eps$ are
uniformly continuous in time with respect to the weak-$*$ topology of
$L^\infty_{\loc}(\R^d)$ (see for instance~\cite[Lemma 8.1.2]{ags:book}) implies that
$u$ belongs to ${\mathcal C}^0 \big( [0,+\infty[ ; L^\infty(\R^d)-w^* \big)$ and $u(0,\cdot)=\bar u$. This proves (i).

{\sc Uniqueness.} Let $u_1$, $u_2 \in L^\infty (p)$ be two
solutions of \eqref{e:cauchy}. We apply Lemma~\ref{l:lift}(iii) with $q=p$ and $\tau=0$ to the functions $u_1 / p$ and $u_2 / p$ (remember Remark~\ref{rmk:wp}), obtaining that, for all $t \in [0,+\infty[$,
$$
p(t,x) \frac{u_1(t,x)}{p(t,x)} = p(t,x) \frac{u_2(t,x)}{p(t,x)} \qquad \text{for a.e.~$x \in \R^d$.}
$$
Recalling again Remark~\ref{rmk:wp}, this gives that for all
$t \in [0,+\infty[$ the equality $u_1(t,x) = u_2(t,x)$ holds for a.e.~$x \in \R^d$, proving (ii).
\end{proof}

\subsection{Well-posedness of strongly continuous solutions}\label{ss:strongtr}

In this section we weaken the assumption $b \in BV_\loc \big( [0,+\infty[ \times\R^d ;\R^d \big)$
replacing it by $b \in BV_\loc \big( ]0,+\infty[ \times\R^d ;\R^d \big)$ and, by adding some
conditions on the density $\rho$ transported by $b$ as in \eqref{amb1}, we can still get
well-posedness of \eqref{e:cauchy} in the class of strongly continuous solutions.

\begin{definition}\label{d:nearly} Let $b : [0,+\infty[ \times \R^d \to \R^d$ be a
bounded vector field. We say that $b$ is
{\em locally strongly nearly incompressible} if there exists a
function $\rho : [0,+\infty[ \times \R^d \to \R$ such that the
following properties are satisfied:
\begin{enumerate}
\item $t \mapsto \rho(t,\cdot)$ is strongly continuous with values in
$L^1_{\loc}(\R^d)$;
\item For every $R>0$ and $T>0$ there exists a
constant $C_{R,T}>0$ such that
\begin{equation}\label{e:rhobound}
\frac{1}{C_{R,T}} \leq \rho(t,x) \leq C_{R,T} \qquad \text{for a.e. $(t,x) \in [0,T] \times B_R(0)$;}
\end{equation}
\item The equation
\begin{equation}\label{e:rhoeq}
\partial_t \rho(t,x) + \div \big( b(t,x) \rho(t,x) \big) = 0
\end{equation}
holds in the sense of distributions on $]0,+\infty[ \times \R^d$.
\end{enumerate}
\end{definition}

\begin{theorem}\label{t:uni2} Let $b \in BV_\loc \big( ]0,+\infty[\times \R^d ;\R^d \big)$
satisfy the assumptions of Definition~\ref{d:nearly}, for some
function $\rho$ which in addition  belongs to
$BV_{\loc}\big(]0,+\infty[ \times \R^d \big)$ and
satisfies $|D\rho(t,\cdot)|(B_R)\in
L^\infty_\loc\big(]0,+\infty[\bigr)$ for all $R>0$. Then there
exists a locally bounded solution $u \in \mathcal C^0 \big(
[0,+\infty[ ; L^1_\loc(\R^d)-s \big)$ to the Cauchy
problem~\eqref{e:cauchy}. Furthermore the solution is unique in this
class.
\end{theorem}

\begin{proof}
{\sc Uniqueness.} Let $u_1$, $u_2 \in \mathcal C^0 \big( [0,+\infty[ ; L^1_\loc(\R^d)-s \big)$
be locally bounded solutions of \eqref{e:cauchy}. We first observe that, from the strong continuity, it follows that
\begin{equation}\label{e:strongdata}
\int_{B_S(0)} | u_1(\eps, x) - u_2(\eps,x) | \, dx \to 0 \qquad \text{ as $\eps \to 0^+$}
\end{equation}
for every $S>0$.

We now apply Lemma \ref{l:lift}(i) with $\beta(z)=|z|$
and $q=\rho$ to the function $u=(u_1-u_2)/\rho$, on the time
interval $[\tau,+\infty[$, for an arbitrary $\tau > 0$, to deduce
that
\begin{equation}\label{e:rr}
\partial_t \big( |u_1 - u_2| \big) + \div \big( b  |u_1 - u_2| \big) = 0
\end{equation}
in the sense of distributions in $]0,+\infty[ \times \R^d$. Arguing as in~\cite[Lemma 3.17]{del:notes},
by \eqref{e:rr} we easily obtain that for any $0< \eps < t \leq T < \infty$ and for any $R>0$, there holds
\begin{equation}\label{e:fin}
\int_{B_R(0)}  | u_1(t,x) - u_2(t,x) | \, dx \leq
\int_{B_{R + T \|b\|_\infty}(0)}  |u_1(\eps, x) - u_2(\eps,x)| \, dx \,.
\end{equation}
By letting $\eps \to 0^+$ and recalling \eqref{e:strongdata}, thanks to \eqref{e:fin}  we get that,
for any $t \in [0,+\infty[$,
$$
u_1(t,x) = u_2(t,x) \qquad \text{ for a.e.~$x\in B_R(0)$.}
$$
Since $R$ is arbitrary, we conclude that $u_1=u_2$.

{\sc Existence.} The proof is organized in three steps.

{\sc Step 1.} We describe the issues we have to address and we sketch how we will proceed in the remaining two steps.

Under the assumptions of the theorem, we can apply Lemma~\ref{l:lift}(ii) to the solution $u/\rho$ on the domain $[\tau, +\infty[ \times \R^d$ for any $\tau >0$ and deduce that, for every locally bounded solution $u$ of \eqref{e:cauchy}, the map
$$
t \mapsto u(t,\cdot)
$$
is strongly continuous from $]0,+\infty[$ with values in $L^1_\loc (\R^d)$. Thus the only issue is constructing a solution which is also strongly continuous at $t=0$.

The strategy is first approximating in $L^1_{\mathrm{loc}}$ our
vector field $b$ with smooth strongly nearly incompressible vector
fields $b_\eps$ with corresponding smooth densities $\rho_\eps$.
Then we observe that, if $z_\eps$ is a solution to the continuity
equation
$$
\left\{
\begin{array}{l}
\p_t z_\eps + \div (b_\eps z_\eps)=0\\
z_\eps(0, \cdot)=\bar u *\eta_\eps\,,
\end{array} \right.
$$
then $\lambda_\eps: =  z_\eps /  \rho_\eps$ is a solution to the
transport equation
$$
\left\{
\begin{array}{l}
\partial_t \l_\eps + b_\eps \cdot \nabla \lambda_\eps=0 \\
\lambda_\eps(0, \cdot) : =\displaystyle \frac{\bar u*\eta_\eps}{\rho_\eps(0, \cdot)}\,.
\end{array} \right.
$$
Noticing that $\rho_\eps(t, \cdot )$ is
strongly continuous by assumption, we only need to prove the
continuity at $0$ of $t \mapsto \l_\eps(t, \cdot)$. This can be
directly shown by using the representation formula for the solution
and exploiting the local strongly nearly incompressibility (see Definition~\ref{d:nearly}). Finally,
we exploit the weak-$\ast$ compactness of $z_\eps$, we let
$\eps \to 0$, and we prove that any limit point of $z_\eps$ provides a strongly
continuous solution.

Another issue we need to address is that the vector field $b$ is only \emph{locally} nearly incompressible, meaning that the bounds~\eqref{e:rhobound} are not necessarily global. We thus introduce a localization argument,
by first restricting to initial data with compact support: because of the finite propagation speed,
the solution we construct has then compact support for every $t \ge 0$. We finally obtain the solution for general initial data by gluing together compactly supported solutions: everything works because what we end up with is a locally finite sum.

{\sc Step 2.} Here, we describe in detail the localization argument.

We introduce a disjoint covering $
\bigcup_{i =1}^{+ \infty} Q_i$ of almost all of $\R^d$,
where $\{ Q_i \}$ are open $d$-dimensional cubes with unit edge length.
Let $\bar u$ as in~\eqref{e:cauchy}, then
$
     \bar u = \sum_{i =1}^{+ \infty}  \bar u_i,
$
where $\bar u_i = \bar u \chi_{ Q_i}$ and $\chi_{ Q_i}$ denotes the characteristic function
of the set $Q_i$. In Step 3 we construct a solution of the Cauchy problem
\begin{equation}
\label{e:cauchy1}
\begin{cases}
\partial_t u_i (t,x) + \div \big( b(t,x) u_i (t,x) \big) = 0 \\
u_i (0,x) = \bar u_i (x)
\end{cases}
\end{equation}
satisfying $u_i  \in \mathcal C^0 \big( [ 0, \, + \infty [ ; L^1 (\R^d)-s \big)$ and
\begin{equation}
\label{e:supp}
    \mathrm{supp} \, u_i  \subseteq [0,T]\times Q_i + \overline{B}_{LT}  (0),
\end{equation}
where $L  = \| b \|_{L^{\infty}}$.
We set
$$
    u (t, \cdot) : =  \sum_{i =1}^{+ \infty}  u_i (t, \cdot).
$$
By noticing that the transport equation is linear and that the previous sum is locally finite by~\eqref{e:supp}, we get that $u \in \mathcal C^0 ( [0, + \infty[ ; L^1_{\mathrm{ loc}} (\R^d)-s)$ is a solution of~\eqref{e:cauchy}.

{\sc Step 3.}
We fix $T > 0$ and we construct a strongly continuous solution $u_i$ of~\eqref{e:cauchy1} defined for $t \in [0, T]$ and satisfying~\eqref{e:supp}. To simplify the notation, we write $u$, $\bar u$ and $Q$ instead of $u_i$, $\bar u_i$ and $Q_i$.

We first approximate $\rho$ and $b$ by $\rho_\eps:=\rho*\eta_\eps$ and
$b_\eps:=(b\rho)*\eta_\eps / \rho_\eps $ respectively. Here $\eta_\eps$ is a
sequence of smooth convolution kernels: the function $\rho_{\eps}$ is always strictly positive by~\eqref{e:rhobound} and hence $b_{\eps}$ is well defined. By using the notation
$L= \|b\|_{L^\infty}$, we get that all vector fields $b_\eps$ are bounded by
$L$.  Also, we can assume that the support of $\bar u \ast \eta_{\eps}$ is strictly contained in $Q  + B_\eps (0)$. Because of~\eqref{e:rhobound}, there exists a constant $M$ satisfying
\begin{equation}
\label{e:emme}
     \frac{1}{M} \leq \rho_{\eps} (t,y) \leq M \qquad \textrm{for every $(t,y)$ such that $t \in [0, T]$ and
     $x \in Q  + B_{1  + 3L T}(0)$}.
\end{equation}
Let $Y_\eps$ denote the flow generated by $b_\eps$ and let
$Z_\eps$ denote the inverse of $Y_\eps$, i.e.~we have
$Z_\eps(t,Y_\eps(t,y))=y$ and $Y_\eps(t,Z_\eps(t,x))=x$. Then
$$
     |Y_\eps(t,y)-y| \leq \|b_\eps\|_{L^\infty} t \leq  Lt
     \qquad \textrm{for every $(t,y) \in [0, T] \times \R^d$}
$$
and analogously
$$
|Z_\eps(t,x)-x| \leq  Lt
     \qquad \textrm{for every $(t, x) \in [0, T] \times \R^d$.}
$$
Let $JY_\eps(t,x)$ denote the Jacobian of $Y_{\eps}(t, x)$: since
$
\rho_\eps(t,Y_\eps(t,y)) JY_\eps(t,y) = \rho_\eps(0,y),
$
by relying on~\eqref{e:emme} we deduce the bound
\begin{equation}\label{e:emme1}
\frac{1}{M^2} \leq JY_\eps(t,y)\leq M^2 \qquad \textrm{for every
$t\in [0,T]$, $y\in Q  + B_{1 + 2LT} (0)$}.
\end{equation}
Since $x\in Q  + B_{1  + L T}(0)$ implies $Z_\eps(t,x)\in Q  + B_{1  + 2L T}(0)$
 we obtain
\begin{equation}
\label{e:bound Zn}
         \frac{1}{M^2} \leq JZ_\eps(t,x) \leq M^2,\qquad
       \qquad \textrm{for every $t\in [0,T]$, $x\in  Q + B_{1 + LT} (0)$}.
\end{equation}
We now define $\bar \l_{\eps} (\cdot) :  =  \bar u \ast \eta_{\eps} / \rho_{\eps} (0, \cdot)$. Then,
by setting
$\l_\eps(t,x): =\bar \l_{\eps}(Z_\eps(t,x))$,
we get a solution of the transport equation
$$
\left\{ \begin{array}{l}
\partial_t \l_\eps + b_\eps\cdot \nabla \l_\eps=0  \\
\l_\eps (0, \cdot) = \bar \l_{\eps} (\cdot) \, .
\end{array} \right.
$$
Since $ \|  \lambda_{\eps} \|_{L^{\infty}}  \leq
\| \bar \lambda_\eps \|_{L^{\infty}} \leq M \| \bar u \|_{L^{\infty}}$,
$\{ \lambda_\eps \}$ is weakly-$\ast$ compact in $L^\infty$.
Moreover, observing that $\rho_\eps (t, \cdot) \to \rho (t, \cdot)$ strongly in $L^1_{\mathrm{loc}}(\R^d)$
for every $t \ge 0$, up to subsequences (not explicitly labelled for simplicity of notation)
we get
$$
\lambda_\eps \rho_\eps \weaks u \qquad \textrm{ weakly-$\ast$ in $L^{\infty} \big( \, [ 0, T] \times \R^d \, \big)$}
$$
for some bounded function $u$ whose support is contained in
$[0, T ] \times \bigl(Q +\overline{B}_{LT} (0)\bigr)$.
First we point out that, since $\lambda_{\eps} \rho_{\eps}$ solves
$$
  \partial_t \big( \lambda_{\eps} \rho_{\eps} \big)
   + \mathrm{div} \big( b_{\eps} \lambda_{\eps} \rho_{\eps} \big)=0 \,,
$$
then $u$ solves
$$
     \partial_t u + \mathrm{div} (bu) =0
$$
in the sense of distributions. Moreover, by extracting a further subsequence we can assume that
\begin{equation}
\label{e:rat}
     \lambda_\eps \rho_\eps (t, \cdot)
     \weaks u (t, \cdot) \qquad
     \textrm{ weakly-$\ast$ in $L^{\infty} \big( \R^d \big)$ for every $t \in [0, T ] \,  \cap \mathbb{Q}$ }.
\end{equation}
Since $ \lambda_\eps \rho_\eps (t, \cdot) $ is weakly-$\ast$ continuous, uniformly in $\eps$
(see for instance \cite[Lemma 8.1.2]{ags:book}),
\eqref{e:rat}  ensures that
$$
    \lambda_\eps \rho_\eps (t, \cdot)
     \weaks u (t, \cdot)
     \qquad \textrm{ weakly-$\ast$ in $L^{\infty} \big( \R^d \big)$ for every $t \in [0, T     ] $ }.
$$
Thus, the only issue we are left to address is that $u(t, \cdot)$ converges strongly in $L^1_\loc(\R^d)$ to $u(0,\cdot) = \bar u$ when $t \to 0^+$. By the lower semicontinuity of the $L^1$ norm with respect to the weak-$\ast$ convergence in $L^\infty$,
\begin{equation}
\label{e:liminf}
          \int_{\R^d} |u(t,x)- \bar u  (x)| \,dx \leq
          \liminf_{\eps \to 0}  \int_{\R^d} |\lambda_\eps \rho_\eps (t,x) -
          \bar \lambda_\eps (x) \rho_\eps (0, x)| \, dx.
\end{equation}
Now, for any given $\eps \leq 1$, we exploit the decomposition
\begin{equation}
\label{e:tri}
\begin{split}
           \int_{\R^d}
           |\lambda_\eps \rho_\eps (t,x)- \bar \lambda_\eps (x)
          \rho_\eps (0, x)|\,dx
&          \leq   \int_{\R^d} \rho_\eps(t, x)  |\l_\eps(t,x)  - \bar \lambda_\eps (x)  | \, dx \\
&       \quad +
          \int_{\R^d} |  \bar \lambda_\eps (x)| \, |  \rho_\eps(t,x)-
            \rho_{\eps}(0,x)|\,dx . \\
\end{split}
\end{equation}
We first focus on the first term in~\eqref{e:tri}: for any
$\psi \in C^\infty_c(\R^d)$  we have
\begin{equation}
\label{e:lambdae}
\begin{split}
                     \int_{\R^d}  |\l_\eps(t,x) & - \bar \lambda_\eps (x)|\,dx
                     = \int_{Q + B_{1 + Lt} (0)}
                    | \bar \lambda_\eps  (Z_\eps(t,x)) - \bar \lambda_\eps (x)| \,dx \\
     &            \leq \int_{Q + B_{1 + Lt} (0) }
                    | \bar \lambda_\eps (Z_\eps(t,x))- \psi(Z_\eps(t,x))| \,dx\\
&                \quad +
                   \int_{Q + B_{1 + Lt} (0) }
                    |\psi(Z_\eps(t,x))- \psi(x)| \,dx
                 +  \int_{Q + B_{1 + Lt} (0)  }
                 |\psi(x)- \bar \lambda_\eps (x)|\,dx\\
                 &  \leq
                 \int_{Q + B_{1 + Lt} (0) }
                | \bar \lambda_\eps (Z_\eps(t,x))- \psi(Z_\eps(t,x))|\,dx +
                 \mathrm{Lip}(\psi) \Leb^{d}\big(Q + B_{1 + Lt} (0) \big)\,L \, t \\
                 & \quad +  \|\psi - \bar \lambda_\eps \|_{L^1 ( \R^d )} . \phantom{\int} \\
\end{split}
\end{equation}
By exploiting the bound~\eqref{e:bound Zn}
on $JZ_\eps$ we then get
\begin{equation}
\label{e:change2}
         \int_{Q  + B_{1 + Lt} (0) }
         | \bar \lambda_\eps(Z_\eps(t,x))- \psi(Z_\eps(t,x))|\,dx \leq
        M^2 \|\psi -  \bar \lambda_\eps \|_{L^1 ( \R^d )}.
\end{equation}
Hence, by combining~\eqref{e:emme},~\eqref{e:lambdae},~\eqref{e:change2} and taking into
account the strong convergence of $\bar\lambda_\eps$ to $\bar u/\rho(0,\cdot)$
as $\eps\to 0$, we eventually obtain
\begin{multline}
\label{e:rolambdae}
 \limsup_{\eps\to 0}\int_{\R^d} \rho_\eps(t, x)  |\l_\eps(t,x)  - \bar \lambda_\eps (x)  | \, dx \\
\leq M(1+M^2) \left\|\psi - \frac{\bar u}{\rho(0,\cdot)} \right\|_{L^1( \R^d ) }
+ M   \mathrm{Lip}(\psi) \Leb^{d}\big(Q + B_{1 + Lt} (0) \big)\,L \, t.
\end{multline}
We now focus on the second term in~\eqref{e:tri}: since $ \bar \lambda_\eps $ is bounded by
$M \| \bar u \|_{L^{\infty}}$, by relying on the properties of the convolution we have
\begin{equation}
\label{e:lazero}
          \int_{\R^d} |  \bar \lambda_\eps (x) | \,
           |  \rho_\eps(t,x)-
            \rho_{\eps}(0,x)|  \,dx
          \leq   M \| \bar u \|_{L^{\infty} }
          \int_{Q+ B_{1}(0) }  |  \rho(t,x)-
              \rho (0,x)|  \,dx
           \leq M \| \bar u  \|_{L^{\infty} } \omega (t).
\end{equation}
In the previous expression, $\omega$ denotes a modulus of
continuity $\omega :[0,+\infty[ \to [0,+\infty[$ for the map
$ t \mapsto \rho(t, \cdot)\chi_{Q+B_1(0)}$, which is strongly continuous
in $L^1 (\R^d)$.

Finally, we combine~\eqref{e:liminf},~\eqref{e:tri},~\eqref{e:rolambdae},~\eqref{e:lazero} and we eventually obtain
\begin{align*}
    \int_{\R^d} |u(t,x)- \bar u  (x)| \,dx \leq  &
    M(1+M^2) \|  \psi - \bar u/\rho(0,\cdot) \|_{L^1(\R^d)} \\
&   + M   \mathrm{Lip}(\psi) \Leb^{d}\big(Q+ B_{1 + Lt} (0) \big)\,L \, t
+  M \| \bar u  \|_{L^{\infty} }  \omega (t). \\
\end{align*}
By letting $t\downarrow 0$ and using the arbitrariness of $\psi$, we
deduce that $t\mapsto u(t,\cdot)$ is strongly continuous
at $0$ in $L^1$. This concludes the proof of Theorem~\ref{t:uni2}.
\end{proof}

We point out that in the proof of Theorem~\ref{t:uni2}, the
assumption that $b$ and $\rho$ are both $BV_{\loc}$ in space and
time is only needed in the Uniqueness Part to show that
distributional solutions are renormalized on $]0,+\infty[ \times
\R^d$ and in the first step of the Existence Part to apply
Lemma~\ref{l:lift}(ii) and get that the map $t \mapsto u(t, \cdot
)$ is strongly continuous in $L^1_\loc (\R^d)$ at any $\tau >0$.
Hence, if we additionally assume that the divergence of $b$ is a
locally integrable function, we can relax the regularity
assumption on both $b$ and $\rho$. More precisely, assume that $b
\in L^1_\loc \big( ] 0,+\infty[ ; BV_\loc(\R^d ;\R^d) \big)$ is a
bounded, locally strongly nearly incompressible vector field such
that $\div b \in L^1_\loc \big( ] 0,+\infty[ \times \R^d \big)$.
Then we can directly apply Ambrosio's theorem~\cite{Amb:trabv} to
deduce that distributional solutions are renormalized. By arguing
as in the proof of Lemma~\ref{l:lift}(ii), we obtain that, if $u$
is a locally bounded solution of~\eqref{e:cauchy}, then $t \mapsto
u(t, \cdot )$ is strongly continuous from $]0,+\infty[$ in
$L^1_\loc(\R^d)$. By combining all the previous considerations, we
get:
\begin{theorem}\label{t:uni3} Let $b \in L^1_\loc \big( ] 0,+\infty[ ; BV_\loc(\R^d ;\R^d) \big)$ be a bounded, locally strongly
nearly incompressible vector field such that $\div b  \in
L^1_{\loc}\big( ] 0,+\infty[ \times \R^d \big)$. Then there exists
a locally bounded solution {$u \in \mathcal C^0 \big( [0,+\infty[
; L^1_\loc(\R^d)-s \big)$} to the Cauchy problem for the
continuity equation~\eqref{e:cauchy}. Furthermore the solution is
unique in this class.
\end{theorem}

\section{An example of nonuniqueness}\label{s2}
Roughly speaking, the results presented in Section~\ref{ss:strongtr} allow to weaken the assumptions on the summability of the $BV$ norm of the vector field, at the price of a restriction of the class of solutions considered. In the model case of a divergence free vector field (which is trivially locally strongly
nearly incompressible), comparing Theorem \ref{t:uni3} and Ambrosio's theorem \cite{Amb:trabv} we see that we swap summability of the $BV$ norm up to $t=0$ for the condition of strong continuity of the solution.

\subsection{Depauw's vector field}\label{ss:depauw} In \cite{Depauw:ex} Depauw constructs an example of nonuniqueness for the transport equation relative to a bounded time-dependent divergence free vector field in the plane. More precisely,  Depauw exhibits
a vector field $a(t,x) \in L^\infty \big( [0,1] \times \R^2 ; \R^2 \big)$, with
$$
\div a=0 \; \text{ and } \; a \in L^1_\loc \big( ] 0,1] ;BV_\loc(\R^2;\R^2)\big) \,,
$$
such that the Cauchy problem
\begin{equation}\label{e:dpcp}
\begin{cases}
\partial_t u + \div ( au ) = 0 \\
u(0,x) = 0
\end{cases}
\end{equation}
has a nontrivial solution $\tilde u \not \equiv 0$. Observe that, being divergence free, the vector field $a$ is locally strongly nearly incompressible (see Definition \ref{d:nearly}): simply take $\rho(t,x) \equiv 1$. Thus, Theorem~\ref{t:uni2} implies that $\tilde u$ cannot be strongly continuous in time, since the constant $0$ trivially provides the unique
 strongly continuous solution of~\eqref{e:dpcp}.

We now want to briefly sketch the construction of \cite{Depauw:ex}. Let us first consider a dyadic subdivision (up to negligible sets) of the time interval, i.e.
$$
[0,1] = \bigcup_{k=1}^\infty I_k \qquad \text{ with } I_k = \left[ \frac{1}{2^k} , \frac{1}{2^{k-1}} \right] \,.
$$
For $t \in I_k$ the vector field $a(t,\cdot)$ does not vary in time, and consists of ``vortexes'' on a pattern with scale $1/2^k$, arranged in such a way that the divergence is zero and, for any compact set $K\subset \R^2$,
\begin{equation}\label{e:dpbv}
\left\| a(t,\cdot) \right\|_{BV(K)} \sim 2^k \qquad \text{ for $t \in I_k$.}
\end{equation}
This allows to construct a solution which undergoes, on each interval $I_k$, a more and more refined mixing, as $t \to 0^+$. More precisely, going backward in time on each interval $I_k$, the solution is rearranged from a function which oscillates between $\pm 1$ on a chessboard of size $1/2^{k-1}$ into a function which oscillates between $\pm 1$ on a chessboard of size $1/2^k$. This in particular implies that
$$
\tilde u (t,\cdot) \weaks 0 \qquad \text{in $L^\infty(\R^2) -w^*$,}
$$
but
$$
\tilde u (t,\cdot) \not \to 0 \qquad \text{in $L^1_\loc(\R^2) -s$,}
$$
which is coherent with Theorem~\ref{t:uni2}.

\subsection{Approximability of the solution with smooth maps} The proof of Theorem~\ref{t:uni2} directly shows that any limit of solutions with smooth approximating vector fields is, as a matter of fact, strongly continuous. Hence we get that the solution $\tilde u$ constructed in \cite{Depauw:ex} cannot be constructed by approximation.

It is also worth to mention a connection with the results in \cite{Bou:cri}, regarding the density of smooth functions in the space of the solutions of the transport equation. Namely, Theorem~2.1 in \cite{Bou:cri} asserts that the approximability of a solution with smooth functions is equivalent to the uniqueness property for both the forward and the backward Cauchy problems. We note that, while the vector field of \cite{Depauw:ex} provides a counterexample to forward uniqueness, the same vector field enjoys uniqueness for the backward Cauchy problem. More comments and variations on the construction by Depauw are also presented in \cite{Bou:cri}.

\subsection{Strong continuity of the vector field does not imply uniqueness}

The vector field described in Section~\ref{ss:depauw} provides a counterexample to the uniqueness in the class of weakly continuous solutions, but
$$
a \not \in {\mathcal C}^0 \big( [0,1] ; L^1_\loc(\R^2;\R^2) \big) \,.
$$
Thus, we could wonder whether Theorem~\ref{t:uni2} holds in the larger class of weakly continuous solutions, provided we impose strong continuity of the vector field, i.e.
$$
b \in {\mathcal C}^0 \big( [0,+\infty[ ; L^1_\loc(\R^d;\R^d) \big)\,.
$$
This means that we are led to the following question.
\begin{question}\label{quest} Let $b \in BV_\loc \big( ]0,+\infty[\times \R^d ;\R^d) \big)$ be a bounded, locally strongly nearly incompressible vector field. Assume that an admissible density $\rho$
in the definition of locally strongly nearly incompressibility belongs to
$BV_{\loc}\big( ]0,+\infty[ \times \R^d \big)$. Assume moreover that
$$
b \in {\mathcal C}^0 \big( [0,+\infty[ ; L^1_\loc(\R^d;\R^d) \big)\,.
$$
Is it true that the Cauchy problem~\eqref{e:cauchy} has a unique locally bounded solution $u$ in the class
${\mathcal C}^0 \big( [0,+\infty[ ; L^\infty_\loc(\R^d) -w^* \big)$?
\end{question}

Notice that, if the answer to Question~\ref{quest} were positive, we would be able to prove one of the results regarding the well-posedness of the chromatography system, the one presented in Theorem~\ref{t:chrom2}, in the more general class of solutions
$$
U = (u_1,u_2) \in {\mathcal C}^0 \big( [0,+\infty[ ; L^\infty_\loc(\R;\R^2) -w^* \big) \,.
$$
However, we can show the following negative result:

\begin{propos}\label{p:osc} The answer to Question~\ref{quest} is negative. More precisely, there exists a vector field $c(t,x) \in L^\infty \big( [0,T] \times \R^2 ; \R^2 \big)$ satisfying the following properties:
\begin{enumerate}
\item[(i)] $c \in BV_\loc \big( ]0,T] \times \R^2 ; \R^2 \big)$;
\item[(ii)] $\div c= 0$ (and thus $c$ is in particular locally strongly nearly incompressible);
\item[(iii)] $c \in {\mathcal C}^0 \big( [0,T] ; L^\infty(\R^2;\R^2) -s \big)$;
\item[(iv)] The Cauchy problem
\begin{equation}\label{e:vtilde}
\begin{cases}
\partial_t v + \div (cv) = 0 \\
v(0,x) = 0
\end{cases}
\end{equation}
has a bounded solution $\tilde v \not \equiv 0$, with $\tilde v \not \in {\mathcal C}^0 \big( [0,T] ; L^1_\loc(\R^2) -s \big)$.
\end{enumerate}
\end{propos}
\begin{proof} We briefly illustrate how it is possible to modify the construction in \cite{Depauw:ex} in such a way that the resulting vector field is strongly continuous. The vector field $a$ is the one described in Section~\ref{ss:depauw}.

{\sc Step 1. Strong continuity at time $t=0$.} The vector field $a$ has length between $0$ and $1$ in all the regions in which it is nonzero. This length is chosen in such a way that the solution is mixed from a chessboard of size $1/2^{k-1}$ into a chessboard of size $1/2^k$ precisely in time $1/2^k$.

To construct the vector field $c$, we decrease the strength of the vector field in each interval in such a way that the discontinuity at $t=0$ is ruled out. First, we substitute each interval $I_k$ with an interval $J_k$ such that
$$
\Leb^1( J_k ) = k 2^{-k} \,.
$$
We notice that $\sum_{k=1}^\infty \Leb^1( J_k ) = T < +\infty$, thus we are again concerned with a finite interval of time. In each interval $J_k$ we then simply multiply the vector field $a$ relative to the interval $I_k$ by $1/k$. Thus, the solution undergoes the same mixing, but in a time which is $k$ times larger than the one in Depauw's example. Notice that we precisely have $\Leb^1(J_k) = k \Leb^1( I_k )$. In this way, we have the convergence $c(t,\cdot) \to 0$ strongly in $L^\infty(\R^2)$ as $t\downarrow 0$.

{\sc Step 2. Blow up of the $BV$ norm.} Let us check that the spatial $BV$ norm of the vector field $c(t,\cdot)$ blows up as $t\downarrow 0$. We first notice that for any compact set $K \subset \R^2$ we have
$$
\|a(t,\cdot)\|_{BV(K)} \sim 2^k \qquad \text{ for $t \in I_k$,}
$$
thus $\|a(t,\cdot)\|_{BV(K)} \sim 1/t$. In our case, we simply have
$$
\|c(t,\cdot)\|_{BV(K)} \sim \frac{2^k}{k} \qquad \text{ for $t \in J_k$.}
$$
Noticing that
$$
\sum_{j=k+1}^\infty \frac{j}{2^j} \geq k \sum_{j=k+1}^\infty \frac{1}{2^j} = \frac{k}{2^k}\,,
$$
we have that $t \geq k/2^k$ when $t \in J_k$, and this implies
$$
\|c(t,\cdot)\|_{BV(K)} \gtrsim \frac{1}{t}\,.
$$
This shows that $c$ does not belong to $L^1\big( [0,T] ; BV_\loc(\R^2;\R^2) \big)$, which is coherent with the uniqueness result of \cite{Amb:trabv}.

{\sc Step 3. Strong continuity at any time $t\in [0,T]$.} Let us notice that in each interval $J_k$ the vector field $c$ is constant with respect to time. However, at the extrema of the intervals $J_k$ the vector field $c$ has a jump discontinuity. To rule out this phenomenon, we choose for every time interval $J_k$ an ``activation function'' $\psi_k (t)$, and we substitute the vector field $c(t,x)$ previously described with $c(t,x)\psi_k(t)$, for $t \in J_k$. We only need to choose, on each $J_k$, the function $\psi_k$ satisfying
$$
\psi_k \in {\mathcal C}^0_c ( \stackrel{\circ}{J_k} ) \,, \qquad  \| \psi_k \|_\infty \leq 2
\quad\text{ and }\quad \int_{J_k} \psi_k(s) \, ds = \Leb^1(J_k) \,.
$$
In this way, the solution undergoes the same mixing as before in each time interval $J_k$, but the discontinuities at the extrema of the intervals $J_k$ are ruled out. It is also immediate that $\partial_t c$ is a locally finite measure,
hence $c\in BV_\loc\bigl(]0,T] \times \R^d\bigr)$.

{\sc Step 4. Conclusion.} In this way, we have constructed a vector field
$c(t,x) \in L^\infty \big( [0,T] \times \R^2 ; \R^2 \big)$ satisfying properties (i)-(ii)-(iii).
The nontrivial solution $\tilde v$ of \eqref{e:vtilde} has the same behaviour of the nontrivial solution
$\tilde u$ of \eqref{e:dpcp} described in Section~\ref{ss:depauw}: at the Lagrangian level we have
simply reparameterized (with the same function) all trajectories,
 ``stretching'' the original vector field $a$ in such a way that we gain strong continuity of the
vector field $c$ with respect to the time.
\end{proof}

\section{Applications to the chromatography system}\label{s:chrom}

\subsection{The chromatography system and classification of the entropies}\label{s:chro:en}
In this section we are concerned with the system~\eqref{e:chrom}
coupled with the initial conditions
\begin{equation}
\label{e:chromdata}
u_1(0,x) = \bar u_1(x) \,, \qquad u_2(0,x) = \bar u_2(x)\,.
\end{equation}
This system belongs to the Temple class \cite{temple} and arises in the study of two component chromatography. Here $u_1$, $u_2 : [0,+\infty[ \times \R \to \R$ are nonnegative functions which express transformations of the concentrations of two solutes, see \cite[page 102]{Bre:book}. Here we make some preliminary considerations by introducing a change of variables and discussing a related classification of the entropies. Our well-posedness results for~\eqref{e:chrom} are discussed in Sections~\ref{s:wellp:w} and \ref{s:wellp:s}.

By introducing the change of variables
\begin{equation}\label{e:change}
v = u_1 + u_2 \,, \qquad w = u_1 - u_2\,,
\end{equation}
system \eqref{e:chrom} becomes
\begin{equation}\label{e:chrom2}
\begin{cases}
\partial_t v + \partial_x \left( \displaystyle \frac{v}{1+v} \right) = 0 \\ \\
\partial_t w + \partial_x \left( \displaystyle \frac{w}{1+v} \right) = 0 \,.
\end{cases}
\end{equation}
Namely, \eqref{e:chrom} splits in the coupling between a scalar
one-dimensional conservation law and a linear continuity equation.
Note that the coefficient of the second equation in
\eqref{e:chrom2} depends on the solution of the first equation,
but the first equation in \eqref{e:chrom2} does not depend on the
solution of the second one.

This decoupling allows for a classification of all the entropies of system \eqref{e:chrom}:

\begin{lemma}\label{l:entropy}
Let $(\eta,q) : \R^2 \to \R \times \R$ be twice continuously differentiable. Then $(\eta,q)$ is an entropy-entropy flux pair for system \eqref{e:chrom} (in the sense of Definition~\ref{e:entrsys}) if and only if the following holds: there exists an entropy-entropy flux pair
$(\tilde \eta, \tilde q) : \R \to \R \times \R$  for
\begin{equation}\label{e:scal}
\partial_t v + \partial_x \left( \displaystyle \frac{v}{1+v} \right) = 0
\end{equation}
satisfying
\begin{equation}\label{e:nuove}
\eta(u_1,u_2) = \tilde \eta (u_1+u_2) + C (u_1 - u_2)\,, \qquad q (u_1,u_2) = \tilde q (u_1+u_2) + \frac{C (u_1-u_2) }{1+u_1+u_2}
\end{equation}
for some real constant $C$. Moreover, $\eta$ is convex if and only if $\tilde \eta$ is convex.
\end{lemma}
\begin{proof} First, observe that $(\eta,q)$ is an entropy-entropy flux pair for system \eqref{e:chrom} if and only if
$$ (\ul, \ulq) (v,w) = (\eta,q)\left( \frac{v+w}{2} , \frac{v-w}{2} \right)$$
is an entropy-entropy flux pair for system \eqref{e:chrom2}. Thus, in the following we focus on system \eqref{e:chrom2}, having flux given by
$$
G(v,w) = \left( \begin{array}{ccc} \displaystyle \frac{v}{1+v}  \\ \\ \displaystyle \frac{w}{1+v} \\ \end{array} \right) \,.
$$
Being $\ul$ an entropy, we have the compatibility condition
\begin{equation}\label{e:compat}
{\rm curl} \, \Big( \nabla \, \ul (v,w) \cdot DG(v,w) \Big) = 0 \,.
\end{equation}
Since
$$
DG (v,w) = \left( \begin{array}{ccc} \displaystyle \frac{1}{(1+v)^2} & 0  \\ \\ \displaystyle \frac{-w}{(1+v)^2} & \displaystyle \frac{1}{1+v} \\ \end{array} \right) \,,
$$
we get from \eqref{e:compat}
$$
{\rm curl} \, \left( \frac{\partial_v \ul -w \partial_w \ul}{(1+v)^2} \; , \; \frac{\partial_w \ul}{1+v} \right) = 0 \,,
$$
namely
$$
v\partial_{vw} \ul + w \partial_{ww} \ul = 0 \,.
$$
By setting $\gamma = \partial_w \ul$, we obtain
$$
\nabla \gamma (v,w) \cdot (v,w) = 0 \qquad\text{for any vector $(v,w)$.}
$$
Hence, $\gamma$ is constant along any ray departing from the origin, and since $\gamma$ is of class $C^1$,
it has to be constant on the whole $\R^2$. Thus,
$$
\ul (v,w) = \tilde \eta (v) + C w
$$
for some real constant $C$. Being in the scalar case, $\tilde \eta$ is trivially an entropy for \eqref{e:scal}: the flux is defined by
$$
\tilde q'(v) = \tilde \eta'(v) g'(v) \,,
$$
where $g(v) = v/(1+v)$ is the flux of \eqref{e:scal}. Moreover
$\ul$ is convex in $(v,w)$ if and only if $\tilde \eta$ is convex
in $v$. Hence, up to constants, the entropy flux $\ulq$ satisfies
$$
\ulq (v,w) = \tilde q (v) + \frac{Cw}{1+v} \,.
$$
This concludes the proof of the first implication: any entropy-entropy flux pair $(\eta,q)$ for \eqref{e:chrom} satisfies \eqref{e:nuove}.

Conversely, a straightforward computation ensures that, for any entropy-entropy flux pair $(\tilde \eta, \tilde q)$ for \eqref{e:scal}, \eqref{e:nuove} defines an entropy-entropy flux pair $(\eta,q)$ for \eqref{e:chrom}.
\end{proof}

Heuristically, Lemma~\ref{l:entropy} can be interpreted as follows. By introducing the change of variables \eqref{e:change} we select a direction $w$ in the phase space $(u_1,u_2)$, along which the solution is simply transported. Consequently, the entropy dissipation in the direction $w$ is zero. Hence, imposing that $(u_1,u_2)$ is an entropy admissible solution of \eqref{e:chrom} means imposing that the entropy dissipates along the orthogonal direction $v$. This interpretation is coherent with the ideas underlying the analysis in \cite{bianchini}.

By relying on Lemma~\ref{l:entropy}, we obtain a one-to-one correspondence between entropy admissible solutions of the system \eqref{e:chrom} and distributional solutions of system \eqref{e:chrom2} such that the first component $v$ is an entropy admissible solution of the scalar conservation law \eqref{e:scal}. Indeed the following holds:

\begin{propos}\label{p:equiv} Let $\bar u_1$, $\bar u_2 \in L^\infty_\loc(\R)$. Then $(u_1 , u_2) \in L^\infty_\loc \big( [0,+\infty[ \times \R ; \R^2 \big)$ is an entropy admissible solution of \eqref{e:chrom} satisfying the initial condition
\begin{equation}
\label{e:inda}
 \big( u_1(0,\cdot) , u_2(0,\cdot) \big) = (\bar u_1, \bar u_2)
 \end{equation}
if and only if $(v,w) = (u_1+u_2,u_1-u_2)$ verifies the following conditions:
\begin{enumerate}
\item $v$ is an entropy admissible solution of the Cauchy problem
\begin{equation}\label{e:cip}
\begin{cases}
\partial_t v + \partial_x \left( \displaystyle \frac{v}{1+v} \right) = 0 \\ \\
v(0,\cdot) = \bar u_1 + \bar u_2 \,;
\end{cases}
\end{equation}
\item $w$ is a solution in the sense of distributions of
\begin{equation}\label{e:ciop}
\begin{cases}
\partial_t w + \partial_x \left( \displaystyle \frac{w}{1+v} \right) = 0 \\ \\
w(0,\cdot) = \bar u_1 - \bar u_2 \,.
\end{cases}
\end{equation}
\end{enumerate}
The initial data in~\eqref{e:cip} and~\eqref{e:ciop} are assumed in the same sense ($L^1_{\loc}-s$ or $L^{\infty}_{\loc}-w^*$) as the datum $(\bar u_1, \bar u_2)$ in~\eqref{e:inda}.
\end{propos}
\begin{proof}
By definition $(u_1 , u_2)$ is an entropy admissible solution of \eqref{e:chrom} if and only if
\begin{equation}\label{e:neve}
\partial_t \eta(u_1,u_2) + \partial_x q(u_1,u_2) \leq 0
\end{equation}
in the sense of distributions for any entropy-entropy flux $(\eta,q)$ with $\eta$ convex. By applying \eqref{e:nuove}, \eqref{e:neve} becomes
\begin{equation}\label{e:pippo}
\partial_t \Big[ \tilde \eta (u_1+u_2) + C (u_1 - u_2) \Big] +
\partial_x \left[ \tilde q (u_1+u_2) + \frac{C (u_1-u_2) }{1+u_1+u_2} \right] \leq 0 \,.
\end{equation}
for any entropy-entropy flux pair $(\tilde\eta,\tilde q)$, with $\tilde\eta$ convex.
Since \eqref{e:neve} is assumed to hold for any entropy-entropy flux $(\eta,q)$ with $\eta$ convex, in~\eqref{e:pippo} we can take any $C \in \R$. Thus, requiring~\eqref{e:pippo} is equivalent to imposing that $w=u_1-u_2$ is a solution
in the sense of distributions of~\eqref{e:ciop} and
$v=u_1+u_2$ is an entropy admissible solution of~\eqref{e:cip}.

Finally, as the change of variables~\eqref{e:change} is linear, the initial data in~\eqref{e:cip}
and~\eqref{e:ciop} are assumed in the same sense as the datum $(\bar u_1, \bar u_2)$ in~\eqref{e:inda}.
\end{proof}

\subsection{Well-posedness results in the class of weakly continuous functions}\label{s:wellp:w}
We first apply Theorem~\ref{t:uni1}, obtaining a well-posedness
result under the following assumptions on the initial datum $\bar
U = (\bar u_1, \bar u_2)$: both $\bar u_1$ and $\bar u_2$ are
nonnegative (which is the physical range), locally bounded, and the
sum $\bar u_1 + \bar u_2$ has locally bounded (but possibly large) total
variation.

We introduce the following set:
\begin{definition}\label{d:F}
We denote by $\F$ the set of functions $U=(u_1,u_2) : \R \to \R \times \R$ such that:
\begin{enumerate}
\item[($\F$1)] $u_1$, $u_2 \in L^\infty_\loc(\R)$;
\item[($\F$2)] $u_1 \geq 0$, $u_2 \ge 0$ a.e.~in $\R$;
\item[($\F$3)] $u_1 + u_2 \in BV_\loc(\R)$.
\end{enumerate}
\end{definition}

We are now ready to introduce our result:
\begin{theorem}\label{t:chrom1} Let $\bar U = (\bar u_1 , \bar u_2) \in
\F.$ Then there exists a unique solution
$$
U = (u_1,u_2) \in {\mathcal C}^0 \big( [0,+\infty[ ;
L^\infty_\loc(\R;\R^2) -w^* \big)
$$
in the sense of distributions of the Cauchy problem
\eqref{e:chrom}-\eqref{e:chromdata}, among those with values in $\F$
and satisfying the entropy
admissibility condition given by Definition~\ref{e:entrsys}.
\\ In addition, this unique solution $U$ satisfies the regularity property
$$
U\in {\mathcal C}^0 \big( [0,+\infty[ ;
L^1_\loc(\R;\R^2) -s \big).
$$
\end{theorem}

\begin{proof}
By applying the change of variables~\eqref{e:change} and by relying on Proposition~\ref{p:equiv}, we reduce the problem to the study of the Cauchy problems~\eqref{e:cip} and~\eqref{e:ciop}.
We then proceed in several steps.

{\sc Step 1.} By applying Theorem~\ref{t:kru} and exploiting
Remark~\ref{r:local} we obtain that the first equation in~\eqref{e:cip} admits a
locally bounded entropy admissible solution $v \in \mathcal C ^0 \big( [0,+\infty[ ; L^{\infty}_\loc (\R) -w^*\big)$. Moreover, thanks to estimate~\eqref{e:comparison}, for every $t>0$ we
have $v(t,x) \ge 0$ for a.e.~$x \in \R$. Estimate~\eqref{e:bvkr} ensures that $v \in L^\infty_\loc \big( [0,+\infty[ ; BV_\loc (\R) \big)$. Finally, by relying on \eqref{e:cip} we get that $v$ has also bounded variation with respect to the time, thus $ v \in BV_\loc \big( [0,+\infty[ \times \R \big)$.

{\sc Step 2.} We now apply Theorem~\ref{t:uni1} with
$$
b(t,x) = \frac{1}{1+v(t,x)}
$$
and $p=v$, deducing existence and uniqueness for the solution $w$ of~\eqref{e:ciop} in
the class
$$
\mathcal C^0 \big( [0,+\infty[ ; L^\infty_\loc (\R) -w^* \big)
\cap L^\infty(v)\,,
$$
where $L^\infty(v)$ is defined as in Definition~\ref{d:spazione}.
Furthermore Theorem~\ref{t:uni1}(iii) implies that $|w(t,x)| \leq v(t,x)$ for a.e.~$x \in\R$,
for all $t \in [0,+\infty[$.

To show that the solutions $u_1$ and $u_2$ are nonnegative, it suffices to observe that we have $u_1 = (v+w)/2$ and $u_2 = (v-w) / 2$, and we have just shown that $|w| \leq v$. This completes the existence part.

{\sc Step 3.} We want to show that the solutions $v(t,\cdot)$ and
$w(t,\cdot)$ constructed above are strongly continuous with values
in $L^1_{\loc}(\R^d)$. The strong continuity of $t \mapsto v(t,\cdot)$ is a
consequence of Theorem~\ref{t:chenra} (since $v \mapsto v/(1+v)$ is uniformly concave we can exploit the considerations in Remark~\ref{r:concave}).
Concerning $w$, it suffices to apply again
Lemma~\ref{l:lift}(ii) with $q = v$ (which has just been shown to be strongly continuous) and $u=w/v$ (see Remark~\ref{rmk:wp}).

{\sc Step 4.} We finally show uniqueness. By the uniform concavity
of the flux function $v \mapsto v / (1+v)$ in \eqref{e:cip},
Theorem~\ref{t:chenra} implies that every weakly continuous
entropy admissible solution of \eqref{e:cip} is in fact strongly
continuous at $t=0$. Hence the uniqueness of $v$ follows from
Theorem~\ref{t:kru}. Concerning $w$, it is sufficient to observe
that, since condition ($\F$2) gives $u_1$ and $u_2$ are
nonnegative, we have
$$
|w|=|u_1-u_2| \leq u_1+u_2=v \,.
$$
Thus $w \in L^\infty(v)$, and we get uniqueness from Step 2.
\end{proof}

As an immediate consequence of Theorem~\ref{t:chrom1}, we obtain:
\begin{corol}\label{c:semigroup1} There exists a unique semigroup
$$
S_t : [0,+\infty[ \; \times \; \F \to \F
$$
satisfying:
\begin{enumerate}
\item[(i)] $S_t$ is continuous with respect to $t$ in $L^\infty_\loc(\R;\R^2)$ endowed with the weak-$\ast$ topology;
\item[(ii)] For every $\bar U = (\bar u_1 , \bar u_2) \in \F$, $S_t \bar U$ provides a distributional solution of the Cauchy problem~\eqref{e:chrom}-\eqref{e:chromdata} satisfying the entropy admissibility condition given by Definition~\ref{e:entrsys}.
\end{enumerate}
\end{corol}
\begin{remark}
By relying on the proof of Theorem~\ref{t:chrom1}, one actually gets a slightly sharper result than the one given in the statement of the theorem. Namely, for any $\bar U = (\bar u_1 , \bar u_2) \in \F$ there exists a unique
$$
U = (u_1,u_2) \in {\mathcal C}^0 \big( [0,+\infty[ ; L^\infty_\loc(\R;\R^2) -w^* \big)
$$
solution in the sense of distributions of the Cauchy problem
\eqref{e:chrom}-\eqref{e:chromdata}, satisfying the entropy
admissibility condition given by Definition~\ref{e:entrsys} and such that  $U(t,\cdot) = \big( u_1(t,\cdot), u_2(t,\cdot) \big)$ verifies  conditions ($\F$1) and ($\F$2) in Definition~\ref{d:F}. For such a solution, property ($\F$3) is automatically satisfied by $U(t,\cdot) = \big( u_1(t,\cdot), u_2(t,\cdot) \big)$ for every $t > 0$.
\end{remark}

\subsection{Well-posedness results in the class of strongly continuous functions}
\label{s:wellp:s}
We now apply Theorem~\ref{t:uni2}, obtaining well-posedness results for~\eqref{e:chrom}-\eqref{e:chromdata} under assumptions on the initial data $(\bar u_1, \bar u_2)$ different from those considered in Section~\ref{s:wellp:w}, and in a different class of solutions. Namely, we relax the regularity assumptions
imposing only that $(\bar u_1, \bar u_2) \in L^{\infty}_{\loc} (\R)$.
The price we have to pay is requiring that the sum $\bar u_1 + \bar u_2$ is
well-separated from $0$ on any compact set, in the sense of~\eqref{e:pappa}.

We first introduce the following set:
\begin{definition}\label{d:G}
We denote by $\G$ the set of functions $U=(u_1,u_2) : \R \to \R \times \R$ such that
\begin{enumerate}
\item[($\G$1)] $u_1$, $u_2 \in L^\infty_\loc(\R)$;
\item[($\G$2)] $u_1 \geq 0$, $u_2 \ge 0$ a.e.~in $\R$;
\item[($\G$3)] for every $R>0$ there exists $\delta_R >0$ such that
\begin{equation}\label{e:pappa}
u_1 + u_2 \geq \delta_R \qquad \text{a.e.~in $]-R,R[$.}
\end{equation}
\end{enumerate}
\end{definition}
We are now ready to state our result:
\begin{theorem}\label{t:chrom2} Let $\bar U = (\bar u_1 , \bar u_2) \in \G$.
 Then there exists a unique
$$
U = (u_1,u_2) \in {\mathcal C}^0 \big( [0,+\infty[ ; L^1_\loc(\R;\R^2) -s\big) \cap L^\infty_\loc \big( [0,+\infty[ \times \R ; \R^2 \big)
$$
solution in the sense of distributions of the Cauchy problem
\eqref{e:chrom}-\eqref{e:chromdata} and
satisfying the entropy admissibility condition given by Definition~\ref{e:entrsys}.
In addition, the unique solution $U$ takes its values in $\G$.
\end{theorem}

\begin{proof}
As in the proof of Theorem~\ref{t:chrom1}, we reduce to study the Cauchy problems~\eqref{e:cip} and~\eqref{e:ciop}. We then proceed in several steps.

{\sc Step 1.} By applying Theorem~\ref{t:kru} and exploiting
Remark~\ref{r:local} (with $f(z)=z/(1+z)$) we obtain that~\eqref{e:cip} admits a
nonnegative entropy admissible solution $v \in L^\infty_\loc\big( [0,+\infty[
\times \R\big)$. Observing that constant functions are entropy
admissible solutions of~\eqref{e:cip}, thanks to
estimate~\eqref{e:comparison} and~\eqref{e:pappa} we easily obtain
that, for every $R$, $t>0$,
\begin{equation}\label{e:pappabuona}
v(t,x) \geq \delta_{R+Lt} \qquad \text{for a.e.~$x \in]-R,R[$.}
\end{equation}
Proposition~\ref{p:reg} combined with Remark~\ref{r:concave} ensures that
$$
v \in L^\infty_\loc \big( ]0,+\infty[ ; BV_\loc (\R) \big) \,,
$$
and using \eqref{e:cip} we deduce that
$$
v \in BV_\loc \big( ]0,+\infty[ \times \R \big) \,.
$$
The fact that
\begin{equation}\label{e:strr}
v \in \mathcal C ^0 \big( [0,+\infty[ ; L^1_\loc (\R) -s \big)
\end{equation}
is a consequence of Theorem~\ref{t:chenra}, thanks to the uniform concavity of $v \mapsto v/(1+v)$.

{\sc Step 2.} Let
\begin{equation}\label{e:b}
b(t,x) = \frac{1}{1+v(t,x)} \,,
\end{equation}
where $v$ is the same function as in Step 1. Then $b$ is a bounded
vector field belonging to $BV_\loc \big( ]0,+\infty[ \times \R
\big)$. Moreover $b$ is locally strongly nearly incompressible:
by taking $\rho = v$ in~\eqref{e:rhobound} and
\eqref{e:rhoeq} we get that Definition~\ref{d:nearly} is satisfied,
as $v$ solves
$$
\partial_t v + \partial_x (bv) = 0
$$
and verifies~\eqref{e:pappabuona} and \eqref{e:strr}. Then, Theorem~\ref{t:uni2} ensures that the Cauchy problem~\eqref{e:ciop} admits a unique solution $w \in {\mathcal
C}^0 \big( [0,+\infty[ ; L^1_\loc(\R) -s \big)$.

{\sc Step 3.} Let $b$ be the same vector field as in~\eqref{e:b}. Then $v$ and $w$ are the unique solutions
(by Theorem~\ref{t:uni2}) of
\begin{equation}\label{e:claim}
\begin{cases}
\partial_t z + \partial_x (bz) = 0 \\
z(0,x) = \bar z (x)\,,
\end{cases}
\qquad \qquad z \in \mathcal{C}^0 \big( [0,+\infty[ ; L^1_\loc(\R)-s \big)\,,
\end{equation}
with $\bar z = \bar v$ and $\bar z = \bar w$ respectively. We claim that the Cauchy
problem~\eqref{e:claim} satisfies the following comparison principle:
\begin{equation}\label{e:claim2}
\bar z  (x) \geq 0 \; \text{for a.e.~$x\in\R$} \qquad \Longrightarrow \qquad
\text{for every $t \ge 0$, } \; z(t,x) \geq 0 \; \text{ for a.e.~$x\in\R$.}
\end{equation}
This implication is proved in Step 4. To conclude with the existence part, we apply~\eqref{e:claim2} with $z=v+w$ and $z=v-w$, noticing that at $t=0$ we have $v+w = 2\bar u_1$ and $v-w = 2\bar u_2$, which are both positive by condition ($\G$2) in Definition \ref{d:G}.

{\sc Step 4.} We now show implication~\eqref{e:claim2}. As it
comes from the existence part in the proof of
Theorem~\ref{t:uni2}, the unique solution of~\eqref{e:claim} is obtained
by approximating the vector field $b$ with smooth vector fields
$b_\eps \to b$ strongly in $L^1_\loc\big( [0,+\infty[ \times \R
\big)$ and considering the solutions $z_\eps$ of the corresponding
problems
$$
\begin{cases}
\partial_t z_\eps + \partial_x (b_\eps z_\eps) = 0 \\
z_\eps(0,x) = \bar z *\eta_\eps(x)\,.
\end{cases}
$$
(With the notation of the proof of Theorem~\ref{t:uni2},
$z_\eps=\l_\eps\rho_\eps$.) Since $\bar z \geq 0$ a.e.~in $\R$,
for every $t>0$ we have $z_\eps(t,x) \geq 0$ for a.e.~$x \in \R$.
Hence, observing that $z_\eps$ converges to $z$ weakly-$\ast$ in
$L^\infty_\loc\big( [0,+\infty[ \times \R \big)$, we conclude that
$z(t,x) \ge 0$ for a.e.~$(t,x) \in [0,+\infty[ \times \R$. Since
$z \in \mathcal{C}^0 \big( [0,+\infty[ ; L^1_\loc(\R) -s \big)$,
implication~\eqref{e:claim2} follows.

{\sc Step 5.} Uniqueness is a simple issue: since $v$ is by assumption strongly continuous, it is unique by Theorem~\ref{t:kru}, while $w$ is unique thanks to the uniqueness part of Theorem~\ref{t:uni2}. \end{proof}

Theorem~\ref{t:chrom2} guarantees, in particular, that the domain $\G$ is invariant for admissible solutions of~\eqref{e:chrom}: if the initial datum $(\bar u_1, \bar u_2)$ belongs to $\G$,
then $\big(u_1(t, \cdot), u_2(t, \cdot) \big) \in \G$ for every $t \ge 0$.
The following is an immediate consequence of Theorem~\ref{t:chrom1}:
\begin{corol}\label{c:semigroup2} There exists a unique semigroup
$$
S_t : [0,+\infty[ \; \times \; \G \to \G
$$
satisfying:
\begin{enumerate}
\item[(i)] $S_t$ is continuous with respect to $t$ in $L^1_\loc(\R;\R^2)$ endowed with the strong topology;
\item[(ii)] For every $\bar U = (\bar u_1 , \bar u_2) \in \G$, $S_t \bar U$ provides a locally bounded distributional solution of the Cauchy problem~\eqref{e:chrom}-\eqref{e:chromdata} satisfying the entropy admissibility condition given by Definition~\ref{e:entrsys}.
\end{enumerate}
\end{corol}

\begin{remark}[Extension to the $k \times k$ system]\label{rmk:k times k}
The $k\times k$ chromatography system is the following:
$$
\partial_t u_i + \partial_x \left( \displaystyle \frac{u_i}{1+u_1+\ldots+u_k} \right) =
0, \qquad i=1,\ldots,k\,.
$$
By doing the change of variable $v=u_1+\ldots+u_k$, $w_i=u_{i+1}$
for $i=1,\ldots,k-1$, our system becomes
$$
\begin{cases}
\partial_t v + \partial_x \left( \displaystyle \frac{v}{1+v} \right) = 0 \\ \\
\partial_t w_i + \partial_x \left( \displaystyle \frac{w_i}{1+v} \right) = 0
\,,\qquad i=1,\ldots,k-1\,.
\end{cases}
$$
All the results proved in this section, including the classification
of the entropies-entropy fluxes given in Lemma \ref{l:entropy}
readily extend to this case.
\end{remark}

\section{Applications to the Keyfitz and Kranzer system}
\label{s:kk}
As a byproduct of our analysis, in this section we apply the results discussed in Section~\ref{s1} to the Keyfitz and Kranzer system
\begin{equation}
\label{e:kk}
\left\{
\begin{array}{lll}
         \partial_t U +\displaystyle{ \sum_{\alpha = 1}^d \frac{\partial }{ \partial x_{\alpha}}
         \Big( f_{\alpha} ( | U |) U \Big) = 0 } \\ \\
         U(0, x) = \bar U (x)\,.
\end{array}
\right.
\end{equation}
The function $U$ takes values in $\R^k$ and depends on $(t, x) \in [0, + \infty [ \, \times \R^d$. For each $\alpha = 1, \dots , d$ the function $f_{\alpha} : \R \to \R$
is smooth. As mentioned in the introduction, to study~\eqref{e:kk} one can first obtain the modulus of the solution $\rho = |U|$ by solving
 \begin{equation}
\label{e:rho}
\left\{
\begin{array}{lll}
         \partial_t \rho +
         \div \big( f (\rho) \rho \big) =0 \\ \\
         \rho (0, x) = | \bar U |(x)\,,
\end{array}
\right.
\end{equation}
where $f= (f_1, \dots, f_d)$. Then, one gets the solution $U=
(\rho \theta_1, \dots , \rho \theta_k)$ by solving the continuity
equations for the components of the ``angular'' part
\begin{equation}\label{e:theta}
    \partial_t \big( \rho \theta_i \big) + \div \big(  f (\rho) \rho \theta_i \big) =0 \,,
    \qquad i=1, \dots, k \,.
\end{equation}
This strategy has been exploited by Ambrosio and De Lellis~\cite{ambdel} and by Ambrosio, Bouchut and De Lellis~\cite{ABDL}, and was inspired by considerations in Bressan~\cite{Bre:illposed}.
Before discussing what we get by combining this strategy with Theorems~\ref{t:uni1} and~\ref{t:uni2}, we need to provide the following definition:
\begin{definition}
\label{d:re}
Let $U$ be a locally bounded function solving~\eqref{e:kk} in the sense of distributions. Then $U$ is a \emph{renormalized entropy solution} if $\rho = |U|$ is an entropy admissible solution of~\eqref{e:rho} (in the sense of Definition~\ref{e:entrsys}) such that ${\lim_{t \to 0^+} \rho (t, \cdot) =|  \bar U|}$ in the strong topology of $L^1_\loc(\R^d)$.
\end{definition}
As pointed out in Ambrosio, Bouchut and De Lellis~\cite{ABDL} by relying on a classification of the entropies for~\eqref{e:kk} due to Frid~\cite{Frid}, under quite general assumptions on the function $f$ any renormalized entropy solution is indeed an entropy admissible solution of~\eqref{e:kk} (see for example Dafermos~\cite[Chapter IV]{Daf:book} for the definition of entropy admissible solution of a system of conservation laws in several space dimensions: this notion is the multidimensional analogue of Definition~\ref{e:entrsys}). A complete proof of this implication can be also found in the notes by De Lellis~\cite[Proposition 5.7]{del:notes}. Conversely, an example discussed in Bressan~\cite[Section 3]{Bre:illposed} shows that, even in the one-dimensional case, in general there might be entropy admissible solutions of~\eqref{e:kk} that are not renormalized entropy solutions. The same example shows that, in general, entropy admissible solutions of~\eqref{e:kk} are not unique.

Existence, uniqueness and stability results for renormalized entropy solutions of~\eqref{e:kk} were obtained in Ambrosio and De Lellis~\cite{ambdel} and Ambrosio, Bouchut and De Lellis~\cite{ABDL}. In particular, in~\cite{ABDL} the Cauchy datum $\bar U$ in \eqref{e:kk} satisfies $|\bar U| \in L^{\infty} \cap BV_\loc$ and $|\bar U|$ can attain the value $0$.

By applying Theorem~\ref{t:uni2}, we manage to relax the assumption $|\bar U| \in L^{\infty} \cap BV_\loc$ by requiring only $|\bar U| \in L^{\infty}$. However, the price we have to pay is that we restrict to the one-dimensional case, we assume that the map $\rho \mapsto f(\rho) \rho$ is uniformly convex, and we impose that the initial datum $|\bar U|$ is well separated from $0$. Similar
results were already known: for example in
Freist{\"u}hler~\cite{freistuhler} and Panov~\cite{Panov:onthe} well-posedness theorems in one
space dimension
for
$L^{\infty}$ initial data were obtained under more general
assumptions than those we consider here. However, Theorem~\ref{t:kk1} below quickly follows from Theorem~\ref{t:uni2}, so for completeness we provide
the details of the proof.
\begin{theorem}
\label{t:kk1}
         Assume that the following conditions are satisfied:
         \begin{enumerate}
         \item there exists $c>0$ such that $\big[ f (\rho) \rho \big]'' \ge c $ on $ \R$;
         \item $\bar U \in L^{\infty} (\R; \R^k)$;
         \item For every $R >0$, there exists $C_R>0$ such that
         $$
             |  \bar U (x) | \ge \frac{1}{C_R} >0 \quad \textrm{for a.e. $x \in [-R, R]$.}
         $$
         \end{enumerate}
         Then the Cauchy problem
         \begin{equation}
\label{e:kko}
\left\{
\begin{array}{lll}
         \partial_t U + \partial_x
         \big[ f ( | U |) U \big] = 0   \\ \\
         U(0, x) = \bar U (x)  \\
\end{array}
\right.
\end{equation}
         admits a unique renormalized entropy solution
        $$
            U \in \mathcal C^0 \big( [0, + \infty[ ; L^1_\loc (\R, \R^k)-s \big)
            \cap L^\infty \big( [0,+\infty[ \times \R ; \R^k \big) \,.
        $$
\end{theorem}
\begin{proof}
{\sc Existence.} The proof exploits the splitting of \eqref{e:kk} in the coupling between \eqref{e:rho} and \eqref{e:theta}. We proceed in several steps.

{\sc Step 1.} We first construct the modulus $\rho$ of the solution. Indeed, Theorem~\ref{t:kru}
ensures that there exists a unique locally bounded entropy admissible solution $\rho$ to
\begin{equation}\label{e:eqrho}
\partial_t \rho + \partial_t \big( f(\rho) \rho \big) = 0
\end{equation}
which satisfies
\begin{equation}\label{e:datorho}
\lim_{t \to 0^+} \rho (t, \cdot) =|  \bar U| \qquad \text{strongly in $L^1_\loc (\R)$.}
\end{equation}

{\sc Step 2.} Our goal is now applying Theorem~\ref{t:uni2}. Let us check that all the hypotheses are satisfied. We first show that
\begin{equation}
\label{e:bkk}
    b(t, x) = \big[ f (\rho) \big] (t, x)
\end{equation}
is a locally strongly nearly incompressible vector field, in the sense of Definition~\ref{d:nearly}.
By combining~\eqref{e:comparison} and conditions (2) and (3) in the statement of the theorem we get that for every $R>0$ and $T>0$ there exists a constant $C_{R, T} >0$ such that
\begin{equation}
\label{eq:bound rho}
\frac{1}{C_{R, T}} \leq \rho(t, x) \leq \| \bar U \|_{L^{\infty}}
\quad \textrm{for a.e. $(t, x) \in \, [0, T] \times B_R(0)$}.
\end{equation}

Also, by Theorem~\ref{t:chenra} and by assumption (1), the map $t \mapsto \rho (t, \cdot) $ is strongly continuous with values in $L^1_\loc (\R)$.

Thanks to Proposition~\ref{p:reg} and assumption (1), we deduce that
$\rho \in L^{\infty}_\loc \big( ]0, + \infty[; BV_\loc (\R) \big)$. Moreover,
by exploiting the equation satisfied by $\rho$
we also have ${\rho \in BV_\loc (]0, + \infty [ \times \R ) }$, and since $f$ is smooth,
the same regularity is inherited by the vector field $b$ defined in~\eqref{e:bkk}.

We can then apply Theorem~\ref{t:uni2} and conclude that, for every $\bar U = (\bar u_1 , \dots , \bar u_k) \in L^\infty(\R;\R^k)$ the Cauchy problems
\begin{equation}
\label{e:ci}
\left\{
\begin{array}{lll}
              \partial_t u_i  + \partial_x ( b u_i ) =0  \\
              u_i (0, x) = \bar u_i ( x ) \,,
\end{array}
\right.
\qquad i=1 , \dots , k\,,
\end{equation}
admit a unique locally bounded solution $U=  (u_1, \dots,u_k ) \in \mathcal C^0 \big( [0, +\infty[ ; L^1_\loc (\R; \R^k)-s \big)$.

{\sc Step 3.} To conclude the proof of the existence part, we need to show that, for every $t \ge 0$,
\begin{equation}
\label{e:mrho}
\rho (t, x) = |U| (t, x) \qquad \textrm{for a.e.~$x \in \R$.}
\end{equation}
Thanks to \eqref{eq:bound rho}, for every $i=1,\ldots,k$ we can define $\theta_i := u_i / \rho$,
and we get
$$
\partial_t \big( \rho \theta_i \big) + \partial_x \big( b \rho \theta_i \big) =0\,,
$$
in the sense of distributions on $]0,+\infty[\times \R$.
Since the functions $\theta_i$ are locally bounded, we can apply Lemma~\ref{l:lift}(i) on the time interval $[\tau,+\infty[$ for an arbitrary $\tau >0$ to deduce that
$$
\partial_t \big( \rho \theta_i^2 \big) + \partial_x \big( b \rho \theta_i^2 \big) =0\,,
$$
in the sense of distributions on $]0,+\infty[\times \R$, or equivalently
\begin{equation}\label{e:toadd}
\partial_t \left( \frac{u_i^2}{\rho} \right) +
\partial_x \left( b \frac{u_i^2}{\rho} \right) =0\,.
\end{equation}
Summing over $i=1,\ldots,k$ the equations in \eqref{e:toadd}, we obtain
\begin{equation}\label{e:eqU}
\partial_t \left( \frac{|U|^2}{\rho} \right) +
\partial_x \left( b \frac{|U|^2}{\rho} \right) =0 \,.
\end{equation}
Moreover, \eqref{e:ci} and \eqref{e:datorho} imply that
\begin{equation}\label{e:caudato}
\frac{|U(0,\cdot)|^2}{\rho(0,\cdot)} = |\bar U|\,.
\end{equation}
Recalling \eqref{e:bkk} and comparing \eqref{e:eqrho} and \eqref{e:eqU}, we see that $\rho$ and $|U|^2/\rho$ are strongly continuous in time, and solve the same Cauchy problem (the initial data coincide because of \eqref{e:caudato}). Thus, by applying Theorem~\ref{t:uni2} we obtain that,
for every $t \ge 0$, $\rho (t, x) = |U| (t, x)$ for a.e.~$x \in \R$, as desired.

{\sc Uniqueness.} Let $U_A$ and $U_B$ be two bounded strongly continuous renormalized entropy solutions of~\eqref{e:kko}. Then $\rho_A = |U_A|$ and $\rho_B= |U_B|$ are two entropy admissible solutions of~\eqref{e:rho}
for which $\lim_{t \to 0^+} \rho_A (t, \cdot) = \lim_{t \to 0^+} \rho_B (t, \cdot) = |  \bar U|$ strongly in $L^1_\loc (\R)$. By the uniqueness part in Theorem~\ref{t:kru}, we deduce that,
for every $t \in [0,+\infty[$,
$$
    \rho_A (t, x) = \rho_B (t, x) \qquad \textrm{for a.e.~$ x \in \R $.}
$$
Let $\rho$ denote the common value of $\rho_A$ and $\rho_B$. We have
$$
\partial_t U_{A,i} + \partial_x \big( f(\rho) U_{A,i} \big) = 0\,, \qquad
\partial_t U_{B,i} + \partial_x \big( f(\rho) U_{B,i} \big) = 0
\qquad \text{for every $i=1,\ldots,k$}
$$
and
$$
U_{A,i} (0,\cdot) = U_{B,i}(0,\cdot) \qquad
\text{a.e.~in $\R$, for every $i=1,\ldots,k$.}
$$
Theorem~\ref{t:uni2} thus gives that for every $i=1,\ldots,k$ and all $t\in [0,+\infty[$
$$
U_{A,i} (t,x) = U_{B,i}(t,x) \qquad
\text{for a.e.~$x \in \R$,}
$$
as desired.
\end{proof}

\begin{remark}
\label{r:sca}
In Theorem~\ref{t:kk1} we restrict to the one-dimensional case $d=1$ because in the proof we need to apply Ole{\u\i}nik's estimate \eqref{e:oleinik} to gain a regularizing effect (this is also the reason why we impose that the function $\rho \mapsto f(\rho) \rho$ is uniformly convex).
\end{remark}

\begin{remark}
\label{r:ABDL} By exploiting the same arguments as in the proof of
Theorem~\ref{t:kk1}, but applying Theorem~\ref{t:uni1} instead of
Theorem~\ref{t:uni2}, one obtains the existence and uniqueness
result proven by Ambrosio, Bouchut and De Lellis in~\cite{ABDL}.
Namely, assume that $|\bar U| \in L^{\infty} \cap BV_\loc(\R^d)$.
Then~\eqref{e:kk} admits a unique renormalized entropy solution $U
\in \mathcal C^0 \big([0, + \infty[; L^{\infty} (\R^d;
\R^k)-w^{\ast} \big).$
\end{remark}

\appendix

\section{Systems of conservation laws in one space dimension} \label{s:con_law:gen}

For completeness, in this short Appendix we go over some results that are used in the paper. Consider a system of conservation laws in one space dimension:
\begin{equation}
\label{e:sy}
            \partial_t  U+ \partial_x \big[F (U)\big] = 0 \, ,
\qquad \text{where $(t, x) \in [0, +\infty[ \times \R$ and $U \in \R^k$,}
\end{equation}
with $F : \R^k \to \R^k$ smooth. For a general introduction to the subject, we refer for example to the books by Bressan~\cite{Bre:book}, by Dafermos \cite{Daf:book} and by Serre \cite{Serre:book}. In the following, we focus on the Cauchy problem, assigning the initial condition
\begin{equation}
\label{e:caudat}
    U(0, x) = \bar U (x).
\end{equation}
It is known that, even if the initial datum $\bar U$ is very regular, in general there is no classical solution to~\eqref{e:sy}-\eqref{e:caudat} defined on the whole time interval $t \in  [0, \, +\infty[$. Examples of solutions starting from a datum $\bar U \in \mathcal C^{\infty}$ and developing discontinuities in finite time are available even in the scalar case $k=1$. It is thus natural to interpret \eqref{e:sy} in the sense of distributions, by requiring that $U$ is a locally bounded measurable function satisfying
\begin{equation}\label{e:distr:sy}
\int_0^{+\infty} \int_{\R} \Big[ U \, \partial_t \Phi   + F(U) \, \partial_x \Phi \Big] \; dx dt = 0 \qquad \text{ for any $\Phi \in \mathcal C^1_c \big( ]0, \, + \infty[ \times \R; \R^k \big)$.}
\end{equation}
However, in general a weak solution of the Cauchy problem~\eqref{e:sy}-\eqref{e:caudat} is not unique. To select a unique solution, various \emph{admissibility conditions} have been introduced, often motivated by physical considerations (see again Dafermos \cite{Daf:book}). Here we focus on the entropy admissibility condition.
\begin{definition}\label{e:entrsys}Let $\eta: \R^k \to \R$ and $q: \R^k \to \R$ be twice continuously differentiable functions. The couple $(\eta, q)$ is an {\em entropy-entropy flux pair} for~\eqref{e:sy} if
 \begin{equation}
 \label{e:entropy:sy}
                              \nabla \eta (U) D F (U) =   \nabla q (U) \qquad \text{ for every $U \in \R^k$,}
\end{equation}
where $DF$ denotes the Jacobian matrix of $F$.  Let $U: [0, \, + \infty [ \times \R \to \R^k$ be a bounded, measurable function. Then $U$ is an \emph{entropy admissible solution} of \eqref{e:sy} if the following holds: for any entropy-entropy flux pair $(\eta, q)$ with $\eta$ convex,
and for any $\phi \in \mathcal C^1_c \big( ]0, \, + \infty[ \times \R \big)$ nonnegative, we have
\begin{equation}
\label{e:e:distr}
         \int_0^{+\infty} \int_{\R} \eta(U) \, \partial_t \phi   + q(U) \, \partial_x \phi \; dx dt   \ge 0.
\end{equation}
\end{definition}

Let us now focus on a single conservation law in one space dimension
\begin{equation}\label{e:scalar}
\partial_t u + \partial_x f(u) = 0 \,.
\end{equation}
Kru\v{z}kov \cite{Kru} proved existence and uniqueness results for
the Cauchy problem in the class of entropy admissible solutions
(in the sense of Definition~\ref{e:entrsys}) assuming the initial
condition in strong sense:

\begin{theorem}[Kru\v{z}kov, \cite{Kru}]\label{t:kru}
Let $f$ be a smooth function and assume that $\bar u \in L^{\infty} (\R)$. Then there exists a unique function $u \in L^\infty \big( [0,+\infty[ \times \R \big)$ satisfying the following properties:
\begin{enumerate}
\item $u$ is  a solution in the sense of distributions of \eqref{e:scalar};
\item the solution $u$ is entropy admissible;
\item $\lim\limits_{t\to 0^+} u(t,\cdot) = \bar u$ in the $L^1_\loc(\R)$ topology.
\end{enumerate}
\end{theorem}
Furthermore, the solution $u$ enjoys the following properties:
\begin{theorem}[Kru\v{z}kov, \cite{Kru}]\label{t:kru2}
Let $f$ be a smooth function, and for $\bar u$, $\bar v \in L^{\infty} (\R)$ let
$u$, $v \in L^\infty \big( [0,+\infty[ \times \R \big)$ be the solutions provided by Theorem~\ref{t:kru}
to the Cauchy problems for \eqref{e:scalar} with initial data $\bar u$ and $\bar v$, respectively.
Define further
\begin{equation}\label{defL}
L:=\sup\Big\{|f'(z)| :\ |z|\leq\max\{\|\bar u\|_\infty,\|\bar v\|_\infty\}\Big\}.
\end{equation}
Then, for any $R>0$ and for all $t>0$ we have
\begin{equation}\label{e:comparison}
\int_{-R}^R \big[ u(t,x) - v(t,x) \big]^+ \, dx \leq \int_{-R-Lt}^{R+Lt} \big[ \bar u (x) - \bar v(x) \big]^+ \, dx \, ,
\end{equation}
where $[ \cdot ]^+$ denotes the positive part. Moreover, if we
assume ${\bar u \in BV_\loc (\R)} $ and denote by
$\mathrm{TotVar}_R \{ u(t, \cdot) \}$ the total variation of the
function $u(t, \cdot)$ on the interval $]-R,R[ \subset \R$, then
for all $t \in [0, + \infty[$
\begin{equation}
\label{e:bvkr}
          \mathrm{TotVar}_{R}  \{ u(t, \cdot) \}  \leq \mathrm{TotVar}_{R+ L t} \{ \bar u \}\,.
\end{equation}
\end{theorem}

Under an assumption of uniform convexity on the flux, namely $f'' (u) \ge c> 0$ for any $u \in \R$, the classical Ole{\u\i}nik's result \cite{Ole} asserts that the solution of the Cauchy problem provided by Theorem~\ref{t:kru}, with initial datum $\bar u \in L^\infty(\R)$, enjoys the following one-sides estimate: for all $t>0$ it holds
\begin{equation}
\label{e:oleinik}
           u(t,y) - u(t,x) \leq \frac{y -x }{c t} \qquad \text{for all $x<y$, $x,\,y\in\R\setminus N_t$,}
\end{equation}
with $N_t$ Lebesgue negligible (possibly depending on $t$).
In particular, \eqref{e:oleinik} implies that the space distributional derivative $\partial_x u(t,\cdot)$
of the solution $u$ satisfies the upper bound
\begin{equation}\label{e:olga}
\partial_x u(t,\cdot) \leq \frac{1}{ct} \qquad \text{in ${\mathcal D}'(\R)$, for all~$t >0$,}
\end{equation}
and this implies in particular that  $\partial_x u(t,\cdot)$ is a measure.
The following fact is well known, but we provide the proof for completeness.
\begin{propos}\label{p:reg}
Let the flux $f$ be uniformly convex, i.e. $f'' \geq c > 0$, and for $\bar u \in L^{\infty} (\R)$
let $u \in L^\infty \big( [0,+\infty[ \times \R \big)$ be the solution provided by Theorem~\ref{t:kru}.
Then $u \in L^\infty_\loc \big( ]0,+\infty[ ; BV_\loc(\R) \big)$.
\end{propos}
\begin{proof}
We decompose the measure $\partial_x u(t,\cdot)$ as
$$
\partial_x u(t,\cdot) = \big[ \partial_x u(t,\cdot) \big]^+ - \big[ \partial_x u(t,\cdot) \big]^- \,.
$$
Then, for every $R>0$ we have
\begin{eqnarray*}
\big| \partial_x u(t,\cdot) \big| ([-R,R]) &=&
\big[ \partial_x u(t,\cdot) \big]^+ ([-R,R]) +  \big[ \partial_x u(t,\cdot) \big]^- ([-R,R]) \\
& = & 2 \big[ \partial_x u(t,\cdot) \big]^+ ([-R,R]) - \big[ \partial_x u(t,\cdot) \big] ([-R,R]) \\
& \leq & \frac{4R}{ct} + | u(R^+) - u(-R^-) |
\; \leq \; \frac{4R}{ct} + 2 \|u\|_\infty \,.
\end{eqnarray*}
\end{proof}

The following results regarding the strong continuity of the entropy admissible solution has been proven by Chen and Rascle \cite{chenra}.
\begin{theorem}\label{t:chenra}
Assume that $f''(u)\geq c >0$ for any $u \in \R$ and let $\bar u \in L^{\infty} (\R)$. Let $u \in L^\infty \big( [0,+\infty[ \times \R \big)$ be a solution in the sense of distributions of \eqref{e:scalar} with $u(0,\cdot)=\bar u$, and assume that $u$ is entropy admissible. Then, $u \in \mathcal C^0 \big( [0,+\infty[ ; L^1_\loc(\R)-s \big)$, and in particular $u$ is the unique solution provided by Theorem~\ref{t:kru}.
\end{theorem}

\begin{remark}\label{r:concave} Proposition~\ref{p:reg} and Theorem~\ref{t:chenra} can be easily extended to the case of a uniformly concave flux, that is $f'' \leq -c <0$. This can be seen by setting $v(t, \, x) = u(t, -x)$ and observing that $v$ is an entropy admissible solution of the equation
$ \partial_t v + \partial_x \big[ (- f)(v) \big] =0$, which has a uniformly convex flux.
\end{remark}

\begin{remark}\label{r:local} If $f'$ is globally bounded, then Theorems~\ref{t:kru} and
\ref{t:kru2} and Proposition~\ref{p:reg} can be extended to the
case of a {\em locally} bounded initial datum $\bar u \in
L^\infty_\loc(\R)$, the only difference in the results obtained
being that now the solution $u$ belongs to $L^\infty_\loc \big(
[0,+\infty[\times \R \big)$. Indeed the constant $L$
in~\eqref{e:comparison} and~\eqref{e:bvkr} actually depends on $\|
\bar u \|_{L^\infty}$ and $\| \bar v \|_{L^\infty}$ only through
$\| f'( \bar u) \|_{L^\infty}$ and $\| f' (\bar v) \|_{L^\infty}$
(see~\eqref{defL}). Being these quantities finite, this ensures a
global bound on the speed of propagation. In addition, observe
that the notions of solution in the sense of distributions and of
entropy admissible solution are both local. Hence, by considering
a suitable truncation of the initial datum and exploiting the
finite propagation speed, we can easily complete the argument.
\end{remark}

\bibliography{bibliochroma}

\end{document}